\documentclass[11pt,reqno]{amsart}
\usepackage{amsmath,amsthm,amssymb}
\usepackage{appendix}

\newtheorem{theorem}{Theorem}[section]
\newtheorem{lemma}[theorem]{Lemma}
\newtheorem{proposition}[theorem]{Proposition}
\newtheorem{corollary}[theorem]{Corollary}

\theoremstyle{definition}

\newtheorem{definition}[theorem]{Definition}

\newtheorem{remark}[theorem]{Remark}

\numberwithin{equation}{section}

\begin{document}

\title[A family of left-invariant SKT metrics $\dots$]{A family of left-invariant SKT metrics on the exceptional Lie group $G_2$}

\author{David N. Pham}
\address{Department of Mathematics $\&$ Computer Science, QCC CUNY, Bayside, NY 11364}
\curraddr{}
\email{dpham90@gmail.com}
\thanks{The author wishes to thank G. Barbaro for helpful discussions.}
%\thanks{}



\dedicatory{}

\begin{abstract}
For a complex manifold $(M,J)$, an SKT (or pluriclosed) metric is a $J$-Hermitian metric $g$ whose fundamental form $\omega:=g(J\cdot,\cdot)$ satisfies the condition $\partial\overline{\partial}\omega=0$.  As such, an SKT metric can be regarded as a natural generalization of a K\"{a}hler metric. In this paper, the exceptional Lie group $G_2$ is equipped with a left-invariant integrable almost complex structure $\mathcal{J}$ via the Samelson construction and a 7-parameter family of $\mathcal{J}$-Hermitian metrics is constructed.  From this 7-parameter family, the members which are SKT are calculated.  The result is a 3-parameter family of left-invariant SKT metrics on $G_2$. As a special case, the aforementioned family of SKT metrics contains all bi-invariant metrics on $G_2$.  In addition, this 3-parameter family of left-invariant SKT metrics are also invariant under the right action of a certain maximal torus $T$ of $G_2$.  Conversely, it is shown that if $g$ is a left-invariant $\mathcal{J}$-Hermitian metric on $G_2$ such that $g$ is invariant under the right action of $T$ and for which $(g,\mathcal{J})$ is SKT, then $g$ must belong to this 3-parameter family of left-invariant SKT metrics.
\end{abstract}

\date{}

\maketitle

\section{Introduction}
Let $(M,J)$ be a complex manifold. A $J$-Hermitian metric $g$ is said to be strong K\"{a}hler with torsion (SKT) (or \textit{pluriclosed}) if the associated fundamental form $\omega:=g(J\cdot,\cdot)$ satisfies 
$$
\partial \overline{\partial}\omega = 0.
$$
It follows from the definition that every K\"{a}hler metric on $(M,J)$ (assuming $(M,J)$ admits K\"{a}hler metrics) is also SKT. However, the converse is not true.  The ``T" in SKT refers to the torsion of the Bismut (or Strominger) connection \cite{Bismut1989,St1986}.  For every $J$-Hermitian metric $g$, there exists a unique connection $\nabla^B$ satisfying $\nabla^BJ=0$, $\nabla^Bg=0$, and whose torsion $T^B$ is totally skew-symmetric in the sense that 
$$
c(X,Y,Z):=g(T^B(X,Y),Z)
$$
is a 3-form.  The connection $\nabla^B$ is the so-called Bismut connection.  One can show that the $3$-form $c$ and the fundamental form $\omega$ are related via
$$
c(X,Y,Z)=d\omega(JX,JY,JZ).
$$
In addition, the SKT condition $\partial\overline{\partial}\omega=0$ can be shown to be equivalent to the condition $dc=0$.  Consequently, $(M,g,J)$ is K\"{a}hler if and only if $T^B=0$.  (In this case, the Bismut connection coincides with the Levi-Civita connection.)  Said differently, the strict SKT metrics on $(M,J)$ (i.e., the ones that are non-K\"{a}hler) all have associated Bismut connections with nonzero torsion.  

Historically, SKT metrics first appeared in the context of physics (e.g., \cite{GHR1984, HS1984}) where the 3-form $c$ induced by the torsion $T^B$ (of what is now called the Bismut connection by differential geometers) is interpreted as the field strength of some potential.  On the mathematical side, there has been considerable interest in SKT metrics precisely because they represent a very natural generalization of K\"{a}hler metrics  (e.g., \cite{DP2023,FG2023, ZZ2023, EFV2012}).  Given existing results in K\"{a}hler geometry, it is only natural to consider how these results generalize in the presence of nonzero torsion. 

One source of examples of SKT metrics come from compact semi-simple Lie groups of even dimension \cite{SSTVP1988}.  The SKT structure is obtained as follows. For a compact Lie group $G$ of even dimension, let $g$ be any bi-invariant metric. (For example, if $K$ is the Killing form of $\mathfrak{g}:=\mbox{Lie}(G)$ and $\lambda>0$, then $-\lambda K$ is positive definite on $\mathfrak{g}$ and the left-invariant Riemannian metric induced by $-\lambda K$ is actually bi-invariant.) Using the Samelson construction \cite{Sam1953}, one can produce a left-invariant integrable almost complex structure $J$ which is compatible with $g$. The pair $(g,J)$ is then an SKT structure on $G$. Of course, if $g$ is only left-invariant instead of bi-invariant, then $(g,J)$ is no longer an SKT structure in general.  This fact serves as motivation to look beyond the bi-invariant case and construct new SKT metrics which are strictly left-invariant. 
 
Recently, in \cite{FG2023}, Fino and Grantcharov studied the existence of left-invariant SKT (as well as Calabi-Yau with torsion (CYT)) metrics on compact semi-simple Lie groups equipped with Samelson complex structures.  In particular, they constructed an explicit 5-parameter family of left-invariant SKT metrics on the Lie group $SO(9)$ which are also invariant under the right action of a certain maximal torus of $SO(9)$. In this paper, we carry out a similar calculation by constructing a 3-parameter family of left-invariant SKT metrics on the exceptional Lie group $G_2$ where $G_2$ is equipped with a Samelson complex structure.  The aforementioned family of left-invariant SKT metrics contains, as a special case, all bi-invariant\footnote{Let $K$ be the Killing form on $\mathfrak{g}_2:=\mbox{Lie}(G_2)$.  Since $G_2$ is a compact simple Lie group, it follows that the space of bi-invariant metrics on $G_2$ is in one-to-one correspondence with the space of (positive definite) inner products on $\mathfrak{g}_2$ of the form $-\lambda K$ for all $\lambda>0$.} metrics on $G_2$.  As was the case in \cite{FG2023}, this 3-parameter family of left-invariant SKT metrics are also invariant under the right action of a certain maximal torus $T$ of $G_2$. Conversely, it is shown that if $g$ is a left-invariant $\mathcal{J}$-Hermitian metric such that $g$ is invariant under the right action of $T$ and for which $(g,\mathcal{J})$ is SKT, then $g$ must belong to this 3-parameter family of left-invariant SKT metrics.

The rest of the paper is organized as follows.  In section \ref{secG2}, we give a brief, yet reasonably self-contained, review of the excpetional Lie group $G_2$.  In section \ref{secHermitianG2}, we review the Samelson complex structure \cite{Sam1953} for even dimensional compact Lie groups and construct a suitable Samelson complex structure $\mathcal{J}$ for $G_2$.  We then proceed to construct a 7-parameter family of left-invariant $\mathcal{J}$-Hermitian metrics on $G_2$.  Lastly, in section \ref{secG2SKT}, we give a brief review of SKT metrics (again keeping things as self-contained as possible) and then proceed to calculate which members of the 7-parameter family of $\mathcal{J}$-Hermitian metrics from section \ref{secHermitianG2} are SKT.  This results in a 3-parameter family of left-invariant SKT-metrics on $G_2$ which are also invariant under the right action of a certain maximal torus $T$ of $G_2$.  We conclude the paper with some brief comments on directions for future work. 

%A source of examples of SKT manifolds were given in \cite{SSTP1988}

\section{The Lie group $G_2$ and its Lie algebra}
\label{secG2}
In this section, we give an elementary introduction to the Lie group $G_2$ and its Lie algebra $\mathfrak{g}_2$ which will be sufficient for the needs of this paper.  Ultimately, we would like to construct an explicit basis of $\mathfrak{g}_2$ as a Lie subalgebra of $\mathfrak{so}(7)$ (the Lie algebra of $7\times 7$ skew-symmetric matrices) which we will use repeatedly in later calculations.  The reader wanting a deeper account of $G_2$ and its Lie algebra $\mathfrak{g}_2$ is referred to \cite{Font2018,Kar2020}.  Some recent work related to $G_2$ is given in \cite{Bar2022} which studied the Aeppli cohomology and pluriclosed flow of compact simply-connected simple Lie groups of rank $2$.
\subsection{Basics on $G_2$}
Let $e_1,e_2,\dots,e_7$ denote the standard basis on $\mathbb{R}^7$ and let $e^1,e^2,\dots, e^7$ denote its dual basis.  Let $e^{ijk}:=e^i\wedge e^j\wedge e^k$ and define $\phi\in \wedge^3(\mathbb{R}^7)^\ast$ via
\begin{equation}
    \label{eqG23form}
    \phi:=e^{147}+e^{257}+e^{367}+e^{123}-e^{156}+e^{246}-e^{345}.
\end{equation}
Likewise, we write $e^{ij}:=e^i\wedge e^j$, $e^{ijkl}:=e^i\wedge e^j\wedge e^k\wedge e^l$ and so on.

We define $G_2$ to be the set of all linear maps $F:\mathbb{R}^7\rightarrow \mathbb{R}^7$ satisfying $F^\ast\phi=\phi$.  Its clear from the definition that if $F_1,F_2\in G_2$, then so is $F_1\circ F_2$.  The following result implies that $G_2$ is also an (abstract) group.
\begin{proposition}
    \label{propG2Invertible}
    If $F\in G_2$, then $F^{-1}\in G_2$ as well.  
\end{proposition}
\begin{proof}
    Let $F\in G_2$ and write $Fe_j=\sum_i F_{ij}e_i$.  Note that 
    $$
        F^\ast e^j=\sum_i F_{ji}e^i.
    $$
    Suppose that $F$ is not invertible. Then for some $c_1,\dots, c_7$ (not all zero) we have
    $$
    c_1Fe_1+c_2Fe_2+\cdots + c_7Fe_7=0.
    $$
    Consider the case where $c_7\neq 0$.  By dividing the above equation by $c_7$, we may assume that $c_7=1$ which implies that 
    $$
    Fe_7=-\sum_{i=1}^6 c_iFe_i.
    $$
    Since $F\in G_2$, we have 
    \begin{align*}
        1&=\phi(e_3,e_6,e_7)=\phi(Fe_3,Fe_6,Fe_7)=-\sum_{i=1}^6c_i\phi(Fe_3,Fe_6,Fe_i)\\
        &=-\sum_{i=1}^6c_i\phi(e_3,e_6,e_i)=-\sum_{i=1}^6c_i\cdot 0=0,
    \end{align*}
    which is is a contradiction (where we note that the fifth equality follows from the definition of $\phi$).  Hence, $F^{-1}$ exists. It then follows immediately from the definition of $G_2$ that $F^{-1}$ also belongs to $G_2$. A similar argument applies to the case where $c_1\neq 0$, $c_2\neq 0$, $\dots$, and $c_6\neq 0$.
\end{proof}
Proposition \ref{propG2Invertible} shows that $G_2\subset GL(\mathbb{R}^7)$. In other words, $G_2$ consists of all $F\in GL(\mathbb{R}^7)$ satisfying $F^\ast\phi=\phi$.  This implies that $G_2$ is a closed subset of $GL(\mathbb{R}^7)$. Since $G_2$ is also a subgroup of $GL(\mathbb{R}^7)$ (at least in the abstract sense), the closed subgroup theorem for Lie groups (cf Theorem 20.10 of \cite{Lee2003}) implies $G_2$ is an embedded Lie subgroup of $GL(\mathbb{R}^7)$.

Until stated otherwise, we will identity a linear map $F:\mathbb{R}^7\rightarrow \mathbb{R}^7$ with its matrix representation with respect to the standard basis $e_1,\dots, e_7$.  Hence, we have an isomorphism $GL(\mathbb{R}^7)\simeq GL(7,\mathbb{R})$ (the Lie group of invertible $7\times 7$ real matrices).  Consequently, we use $GL(\mathbb{R}^7)$ and $GL(7,\mathbb{R})$ interchangeably.  

We will now show that $G_2$ is a connected Lie subgroup of $SO(7)$ (the Lie group of $7\times 7$ orthogonal matrices with determinant $1$).  The statement that $G_2$ is connected is easy to see.  Indeed, let $F\in G_2$ and let 
$$
F_t:=(1-t)F+t\mathbf{1},~t\in \mathbb{R},
$$
where $\mathbf{1}:\mathbb{R}^7\rightarrow \mathbb{R}^7$ is the identity map.  Then
$$
F_t^\ast\phi=(1-t)F^\ast\phi+t\mathbf{1}^\ast\phi=(1-t)\phi+t\phi=\phi.
$$
Hence, $F_t$ is a (smooth) curve in $G_2$ from $F$ to $\mathbf{1}$.  

To show that $G_2\subset SO(7)$, first consider the expression
$$
(\iota_u\phi)\wedge (\iota_v\phi)\wedge \phi\in \wedge^7(\mathbb{R}^7)^\ast,
$$
where $u,v\in \mathbb{R}^7$ and $\iota_u\phi:=\phi(u,\cdot,\cdot)$ is the interior product.  By direct calculation, one verifies that 
\begin{equation}
\label{eqPhiIdentity1}
\iota_{e_i}\phi \wedge \iota_{e_j}\phi\wedge \phi =0,~\forall~i\neq j.
\end{equation}
For example, for $i=2$ and $j=5$, one has
$$
\iota_{e_2}\phi=e^{57}-e^{13}+e^{46}
$$
$$
\iota_{e_5}\phi=-e^{27}+e^{16}-e^{34},
$$
Then
\begin{align*}
\nonumber
\iota_{e_2}\phi\wedge \iota_{e_5}\phi&=e^{5716}-e^{5734}+e^{1327}-e^{4627}.
\end{align*}
Now if $e^{ijk}$ is any term in $\phi$, one has
\begin{align*}
    e^{5716}\wedge e^{ijk}&=0,\\
    e^{5734}\wedge e^{ijk}&=0,\\
    e^{1327}\wedge e^{ijk}&=0,\\
    e^{4627}\wedge e^{ijk}&=0,
\end{align*}
which implies $\iota_{e_2}\phi\wedge \iota_{e_5}\phi\wedge \phi =0$.  In a similar fashion, one verifies that 
\begin{equation}
\label{eqPhiIdentity2}
\iota_{e_i}\phi \wedge \iota_{e_i}\phi \wedge \phi = -6\mu_0,~\forall~i
\end{equation}
where $\mu_0:=e^1\wedge e^2\wedge \cdots \wedge e^7$.  From (\ref{eqPhiIdentity1}) and (\ref{eqPhiIdentity2}), we have
\begin{equation}
    \label{eqPhiIdentity3}
    \iota_{u}\phi \wedge \iota_{v}\phi \wedge \phi = -6\langle u,v\rangle \mu_0,~\forall~u,v\in \mathbb{R}^7,
\end{equation}
where $\langle\cdot,\cdot\rangle$ is the standard inner product on $\mathbb{R}^7$.
\begin{proposition}
    \label{propG2SO7}
    For all $F\in G_2$, $\langle Fu,Fv\rangle = \langle u,v\rangle$ and $F^\ast \mu_0=\mu_0$. In other words, $G_2\subset SO(7)$.
\end{proposition}
\begin{proof}
    Let $F\in G_2$ and apply the dual map $F^\ast$ to both sides of (\ref{eqPhiIdentity3}):
    \begin{equation}
    \label{eq1propG2SO7}
    F^\ast(\iota_{u}\phi) \wedge F^\ast(\iota_{v}\phi) \wedge F^\ast\phi = -6\langle u,v\rangle F^\ast\mu_0.
    \end{equation}
    Since $F^\ast(\iota_u\phi)=\iota_{F^{-1}u}\phi$ and $F^\ast \phi=\phi$, we have
    \begin{equation}
    \label{eq2propG2SO7}
    (\iota_{F^{-1}u}\phi) \wedge (\iota_{F^{-1}v}\phi) \wedge \phi = -6\langle u,v\rangle F^\ast\mu_0.
    \end{equation}
    Using the identity (\ref{eqPhiIdentity3}) on the left side of (\ref{eq2propG2SO7}), we see that
    \begin{equation}
    \label{eq3propG2SO7}
       -6\langle F^{-1}u,F^{-1}v\rangle\mu_0 =  -6\langle u,v\rangle F^\ast\mu_0=-6\det(F)\langle u,v\rangle\mu_0.
    \end{equation}
    Hence,
    \begin{equation}
        \label{eq4propG2SO7}
        \langle F^{-1}u,F^{-1}v\rangle=\det(F)\langle u,v\rangle.
    \end{equation}
    Replacing $u$ with $Fu$ and $v$ with $Fv$, we have
    \begin{equation}
        \label{eq5propG2SO7}
        \langle u,v\rangle=\det(F)\langle Fu,Fv\rangle.
    \end{equation}
    Identifying $F$ with its matrix representation with respect to the standard basis $e_1,e_2,\dots, e_7$, (\ref{eq5propG2SO7}) can be rewritten as
    \begin{equation}
        \label{eq6propG2SO7}
        u^Tv=\det(F)u^TF^TFv.
    \end{equation}
    This implies 
    \begin{equation}
        \label{eq7propG2SO7}
        \det(F)F^TF=1_{7},
    \end{equation}
    where $1_7$ is the $7\times 7$ identity matrix.  Taking the determinant on both sides of (\ref{eq7propG2SO7}) gives
    \begin{equation}
        \label{eq8propG2SO7}
        \det(F)^7\det(F^TF)=\det(F)^9=1.
    \end{equation}
    Hence $\det(F)=1$.  Substituting this into (\ref{eq5propG2SO7}) gives $\langle Fu,Fv\rangle =\langle u,v\rangle$. 
\end{proof}

Before concluding this section, we give the motivation for $\phi$ in (\ref{eqG23form}).  In short, $\phi$ is simply a concise way of encoding the $7$-dimensional cross product on $\mathbb{R}^7$.  Explicitly, the cross product $\times:\mathbb{R}^7\times \mathbb{R}^7\rightarrow \mathbb{R}^7$ is the bilinear map defined by
$$
\langle u\times v, w\rangle = \phi(u,v,w),~\forall~u,v,w\in \mathbb{R}^7.
$$
Since $\phi$ is a 3-form, it follows that 
$$
u\times v=-v\times u
$$
and $\langle u\times v, u\rangle =\langle u\times v, v\rangle =0$, that is, $u\times v$ is orthogonal to $u$ and $v$. One can check in a straightforward  (albeit somewhat tedious) way that
\begin{align}
    \label{eqCross1}
    |u\times v|^2&=|u|^2|v|^2-\langle u,v\rangle^2\\
    \label{eqCross2}
    u\times (u\times v)&=-|u|^2v+\langle u,v\rangle u,
\end{align}
where $|u|^2:=\langle u,u\rangle$. As a side remark, note that (\ref{eqCross2}) implies that the cross product induces an almost complex structure $J$ on the $6$-sphere as follows: for $p\in S^6$,
$$
J_p: T_pS^6\rightarrow T_pS^6,~v\mapsto p\times v.
$$
We conclude this section with the following result which expresses the definition of $G_2$ in terms of the cross product $\times$.
\begin{proposition}
    Let $F\in GL(\mathbb{R}^7)$.  Then $F\in G_2$ if and only if $F(u\times v)=(Fu)\times (Fv)$ for $u,v\in\mathbb{R}^7$.
\end{proposition}
\begin{proof}
    Suppose $F\in G_2$.  Since $F$ preserves the inner product $\langle\cdot,\cdot\rangle$ (by Proposition \ref{propG2SO7}), we have
    $$
    \langle F(u\times v),Fw\rangle=\langle u\times v,w\rangle=\phi(u,v,w)
    $$
    for all $u,v,w\in \mathbb{R}^7$.  On the other hand,
    $$
    \phi(Fu,Fv,Fw)=\langle (Fu)\times (Fv),Fw\rangle.
    $$
    Since $F^\ast\phi =\phi$, we have
    $$
    \langle F(u\times v),Fw\rangle=\langle (Fu)\times (Fv),Fw\rangle
    $$
    which implies $F(u\times v)=(Fu)\times (Fv)$.

    Now suppose that $F(u\times v)=(Fu)\times (Fv)$.  First, we show that $Fe_i$ and $Fe_j$ are orthogonal for $i\neq j$. From the definition of the cross product, we have $e_i\times e_j=\pm e_k$ for some $k$.  By swapping $i$ and $j$ if necessary, we may assume that $e_i\times e_j=e_k$.  The definition of the cross product implies that $e_j\times e_k=e_i$ and $e_k\times e_i = e_j$. From this, we have
    \begin{align*}
        \langle Fe_i,Fe_j\rangle &=\langle F(e_j\times e_k),Fe_j\rangle\\
        &=\langle (Fe_j)\times (Fe_k),Fe_j\rangle\\
        &=\phi(Fe_j,Fe_k,Fe_j)\\
        &=0.
    \end{align*}
    We now show that $|Fe_i|=1$.  Using (\ref{eqCross1}), we have
    \begin{align*}
        |Fe_i|^2&=\langle Fe_i, Fe_i\rangle\\
        &=\langle F(e_j\times e_k), F(e_j\times e_k)\rangle\\
        &=\langle (Fe_j)\times (Fe_k), (Fe_j)\times (Fe_k)\rangle \\
        &=|(Fe_j)\times (Fe_k)|^2\\
        &=|Fe_j|^2|Fe_k|^2-\langle Fe_j,Fe_k\rangle^2\\
        &=|Fe_j|^2|Fe_k|^2,
    \end{align*}
   where the last equality follows from the previous calculation which showed that $Fe_j$ and $Fe_k$ are orthogonal for $j\neq k$.

   Taking the square root of the last equality and using the fact that $e_i=e_j\times e_k$ holds when taking cyclic permutations of $i$,$j$, and $k$, we obtain
   \begin{align*}
       |Fe_i|&=|Fe_j||Fe_k|\\
       |Fe_j|&=|Fe_k||Fe_i|\\
       |Fe_k|&=|Fe_i||Fe_j|.
   \end{align*}
   This implies 
   $$
   |Fe_j|=|Fe_i|^2|Fe_j|
   $$
   from which we conclude that $|Fe_i|=1$.  Combining this with the previous fact that $Fe_i$ and $Fe_j$ are orthogonal for $i\neq j$ gives
   $$
    \langle Fe_i, Fe_j\rangle =\delta_{ij}.
   $$
   This in turn implies 
   \begin{equation}
   \label{eqPropG2Cross}
    \langle Fu,Fv\rangle =\langle u,v\rangle
   \end{equation}
   for all $u,v\in \mathbb{R}^7$.  Using this, we have
   \begin{align*}
       \phi(Fu,Fv,Fw)&=\langle (Fu)\times (Fv),Fw\rangle\\
       &=\langle F(u\times v),Fw\rangle \\
       &=\langle u\times v,w\rangle\\
       &=\phi(u,v,w),
   \end{align*}
   where the second to last equality follows from (\ref{eqPropG2Cross}).  Hence, $F\in G_2$.  
\end{proof}

\subsection{Calculations on $\mathfrak{g}_2:=\mbox{Lie}(G_2)$}\label{Secg2Calc} 
Let $X\in \mathfrak{gl}(7,\mathbb{R})$ and let $A(t):=\mbox{exp}(tX)$.  Then $X\in \mathfrak{g}_2$ if and only if 
$$
\frac{d}{dt}\Big|_{t=0}A(t)^\ast \phi = 0.
$$
Let $A_{ij}(t)$ denote the $(i,j)$-element of $A(t)$. Expanding $A^\ast\phi$ gives
\begin{align*}
    A^\ast\phi&=\sum_{k,l,m}\left[A_{1k}A_{4l}A_{7m}+A_{2k}A_{5l}A_{7m}+A_{3k}A_{6l}A_{7m}+A_{1k}A_{2l}A_{3m}\right]e^{klm}\\
    &+\sum_{k,l,m}\left[-A_{1k}A_{5l}A_{6m}+A_{2k}A_{4l}A_{6m}-A_{3k}A_{4l}A_{5m}\right]e^{klm}.
\end{align*}
Differentiating at $t=0$ and using the fact that $A_{ij}(0)=\delta_{ij}$ and $\frac{dA_{ij}}{dt}(0)=X_{ij}$ gives
\begin{align*}
    \frac{d}{dt}\Big|_{t=0}A^\ast \phi&=X_{1k}e^{k47}+X_{4l}e^{1l7}+X_{7m}e^{14m}\\
    &+X_{2k}e^{k57}+X_{5l}e^{2l7}+X_{7m}e^{25m}\\
    &+X_{3k}e^{k67}+X_{6l}e^{3l7}+X_{7m}e^{36m}\\
    &+X_{1k}e^{k23}+X_{2l}e^{1l3}+X_{3m}e^{12m}\\
    &-X_{1k}e^{k56}-X_{5l}e^{1l6}-X_{6m}e^{15m}\\
    &+X_{2k}e^{k46}+X_{4l}e^{2l6}+X_{6m}e^{24m}\\
    &-X_{3k}e^{k45}-X_{4l}e^{3l5}-X_{5m}e^{34m},
\end{align*}
where we have invoked the Einstein summation convention of summing over repeated indices in a term to simplify notation.

Note that $\wedge^3(\mathbb{R}^7)^\ast$ is a 35-dimensional vector space with basis
$$
\{e^{ijk}~|~1\le i<j<k\le 7\}.
$$
The condition $\frac{d}{dt}\Big|_{t=0}A(t)^\ast \phi = 0$ yields 35 equations where each equation is the coefficient of some basis element $e^{ijk}$ set equal to zero.  Seven of the equations are
\begin{align*}
&X_{11}+X_{22}+X_{33}=0,~X_{11}+X_{44}+X_{77}=0,\\
&-X_{33}-X_{44}-X_{55}=0,~-X_{11}-X_{55}-X_{66}=0,\\
&-X_{11}-X_{55}-X_{66}=0,~X_{22}+X_{44}+X_{66}=0,\\
&X_{22}+X_{55}+X_{77}=0,
\end{align*}
which imply $X_{ii}=0$ for $i=1,2,\dots, 7$.

The remaining 28 equations form seven groups of four equations.  For example, one of the seven groups of four equations is
\begin{align*}
&X_{45}+X_{21}-X_{67}=0,~X_{12}+X_{54}+X_{67}=0,\\
&X_{76}-X_{12}+X_{45}=0,~X_{76}-X_{54}+X_{21}=0.
\end{align*}
Solving these seven groups of four equations we see that $X\in \mathfrak{g}_2$ if and only if $X_{ji}=-X_{ij}$ (this condition is not surprising since $G_2\subset SO(7)$) and the following seven constraints are satisfied:
\begin{align*}
    &X_{76}=X_{54}+X_{12},~X_{61}=X_{72}-X_{34},\\
    &X_{62}=X_{53}+X_{17},~X_{14}=X_{63}+X_{52},\\
    &X_{51}=X_{42}+X_{37},~X_{64}=X_{31}-X_{75},\\
    &X_{74}=X_{23}-X_{56}.
\end{align*}
An arbitrary $7\times 7$ skew-symmetric matrix $X$ has 21 free parameters, namely, 
$$
\{X_{ij}~|~1\le i<j\le 7\}.
$$
The constraints given above reduces the number of free parameters by $7$.  Hence, 
$$
\dim \mathfrak{g}_2=21-7=14.
$$
A basis for $\mathfrak{so}(7)$ is 
$$
\{E_{ij}~|~1\le i<j\le 7\},
$$
where $E_{ij}$ is the matrix whose $(i,j)$-element is $1$, $(j,i)$-element is $-1$, and whose other elements are all zero.  From the additional seven conditions given above, we obtain a basis for $\mathfrak{g}_2$:
\begin{align}
    \label{eqG2basis1}
    &b_1=E_{12}-E_{67},~b_2=E_{13}+E_{46},~b_3=E_{17}-E_{26},\\
    \label{eqG2basis2}
    &b_4=E_{23}-E_{47},~b_5=E_{24}+E_{15},~b_6=E_{25}-E_{14},\\
    \label{eqG2basis3}
    &b_7=E_{27}+E_{16},~b_8=E_{34}+E_{16},~b_9=E_{35}+E_{26},\\
    \label{eqG2basis4}
    &b_{10}=E_{36}-E_{14},~b_{11}=E_{37}-E_{15},~b_{12}=E_{45}+E_{67},\\
    \label{eqG2basis5}
    &b_{13}=E_{56}+E_{47},~b_{14}=E_{57}-E_{46}.
\end{align}
For the convenience of the reader (and the apparent difficulty of finding such information in the literature), the nonzero bracket relations of $\mathfrak{g}_2$ with respect to this basis is given in Appendix \ref{AppendixG2Bracket}.

\section{Left-invariant Hermitian structures on $G_2$}
\label{secHermitianG2}
\subsection{Samelson Complex Structure}
It was shown by Samelson in \cite{Sam1953} (and independently by Wang in \cite{Wang1954}) that every compact even dimensional Lie group admits a left-invariant integrable almost complex structure.  We briefly review this construction.  

Let $G$ be a compact even dimensional Lie group and let $\mathfrak{g}:=\mbox{Lie}(G)$ be the Lie algebra of left-invariant vector fields on $G$ (which we identify with the tangent space at the identity $T_eG$).  Let $\mathfrak{t}$ be any maximal abelian Lie subalgebra of $\mathfrak{g}$, let $\mathfrak{g}_{\mathbb{C}}$ and $\mathfrak{t}_{\mathbb{C}}$ denote the complexificaton of $\mathfrak{g}$ and $\mathfrak{t}$ respectively, and let 
$$
\mathfrak{g}_{\mathbb{C}}=\mathfrak{t}_{\mathbb{C}}\oplus \bigoplus_{\alpha\in \Delta}\mathfrak{g}_\alpha
$$
be the root space decomposition associated to $\mathfrak{t}$ where $\Delta\subset (\mathfrak{t}_{\mathbb{C}})^\ast\backslash \{0\}$ are the set of nonzero roots and  
$$
\mathfrak{g}_{\alpha}:=\{X\in \mathfrak{g}_{\mathbb{C}}~|~\mbox{ad}_HX= \alpha(H)X,~\forall~H\in \mathfrak{t}_{\mathbb{C}}\}
$$
is the root space associated to $\alpha: \mathfrak{t}_\mathbb{C}\rightarrow \mathbb{C}$. (Note that $\mathfrak{g}_0=\mathfrak{t}_{\mathbb{C}}$ is the root space associated to $0\in (\mathfrak{t}_\mathbb{C})^\ast$.)
From the definition, we immediately have
$$
[\mathfrak{g}_{\alpha},\mathfrak{g}_\beta]\subset \mathfrak{g}_{\alpha+\beta}
$$
if $\alpha+\beta\in \Delta\cup\{0\}$.  Otherwise, $[\mathfrak{g}_{\alpha},\mathfrak{g}_\beta]=\{0\}$. In addition, for $\alpha\in R$, we note that $\mathfrak{g}_{\alpha}$ and $\mathfrak{g}_{-\alpha}$ are complex conjugates of one another:
$$
\mathfrak{g}_{-\alpha}=\overline{\mathfrak{g}_{\alpha}}.
$$
With additional work, one can show that $\dim_{\mathbb{C}}\mathfrak{g}_{\alpha}=1$ for all $\alpha\in \Delta$.  Moreover, $\Delta$ can be expressed as a union of \textit{positive} and \textit{negative} roots:
$$
\Delta=\Delta^+\cup \Delta^-
$$
where $\Delta^+\cap \Delta^-=\emptyset$, $\Delta^-=-\Delta^+$, and $\Delta^+$ has the following property: if $\alpha,\beta\in \Delta^+$ and $\alpha+\beta\in \Delta$, then $\alpha+\beta\in \Delta^+$ (see e.g. \cite{Sep2007}).  With this information, the root space decomposition can be rewritten as 
$$
\mathfrak{g}_{\mathbb{C}}=\mathfrak{t}_{\mathbb{C}}\oplus \bigoplus_{\alpha\in \Delta^+}(\mathfrak{g}_\alpha\oplus \overline{\mathfrak{g}_{\alpha}}).
$$

Since we are constructing a left-invariant integrable almost complex structure, the problem can be solved entirely at the Lie algebra level.  Namely, we seek a linear map $\mathcal{J}:\mathfrak{g}\rightarrow \mathfrak{g}$ such that $\mathcal{J}^2=-\mbox{id}$ and $N_{\mathcal{J}}(X,Y)=0$ for all $X,Y\in \mathfrak{g}$, where $N_{\mathcal{J}}$ is the Nijenhuis tensor associated to $\mathcal{J}$ which is given by
$$
N_{\mathcal{J}}(X,Y):=\mathcal{J}[\mathcal{J}X,Y]+\mathcal{J}[X,\mathcal{J}Y]+[X,Y]-[\mathcal{J}X,\mathcal{J}Y].
$$
$\mathcal{J}$ then determines a left-invariant integrable almost complex structure $J$ on $G$ via
$$
J_g:=(l_g)_\ast\circ \mathcal{J}\circ (l_{g^{-1}})_\ast,\hspace*{0.1in}\forall~g\in G,
$$
where $l_g: G\rightarrow G$ is left translation by $g\in G$ and $(l_g)_\ast$ is the pushforward on the tangent bundle.  In particular, $J_e=\mathcal{J}$.  (More often than not, we will not make a distinction between $\mathcal{J}$ and the left-invariant almost complex structure $J$ determined by $\mathcal{J}$.)

The Samelson construction which produces a linear map $\mathcal{J}:\mathfrak{g}\rightarrow \mathfrak{g}$ with the above properties is defined as follows:  we first define a map $\mathcal{J}':\mathfrak{g}_{\mathbb{C}}\rightarrow \mathfrak{g}_{\mathbb{C}}$ on the complexified Lie algebra $\mathfrak{g}_{\mathbb{C}}$ which satisfies the following conditions:
\begin{itemize}
\item[(1)] Choose $\mathcal{J}'$ so that $\mathcal{J}'\mathfrak{t}\subset \mathfrak{t}$ and $\mathcal{J}'|_{\mathfrak{t}}$ squares to negative the identity.  Since $\mathfrak{g}$ is even dimensional, it follows that $\mathfrak{t}$ is also even dimensional.  So this condition can always be satisfied.
\item[(2)] For $\alpha\in \Delta^+$ and $E_\alpha\in \mathfrak{g}_{\alpha}$, define 
$$
\mathcal{J}'E_{\alpha} := iE_{\alpha},\hspace*{0.1in}\mathcal{J}'\overline{E}_{\alpha} := -i\overline{E}_{\alpha}.
$$
\end{itemize}
Clearly, the linear map $\mathcal{J}':\mathfrak{g}_{\mathbb{C}}\rightarrow \mathfrak{g}_{\mathbb{C}}$ squares to negative the identity.  By a straightforward calculation, we also see that $\mathcal{J}'\mathfrak{g}\subset \mathfrak{g}$.  Let $\mathcal{J}:=\mathcal{J}'|_{\mathfrak{g}}:\mathfrak{g}\rightarrow \mathfrak{g}$.  Then clearly $\mathcal{J}^2=-\mbox{id}_{\mathfrak{g}}$. With some additional work, one shows that the Nijenhuis tensor associated to $\mathcal{J}':\mathfrak{g}_{\mathbb{C}}\rightarrow \mathfrak{g}_{\mathbb{C}}$ also vanishes:
$$
N_{\mathcal{J}'}(X,Y)=0,~\forall~X,Y\in \mathfrak{g}_{\mathbb{C}}.
$$
Since $\mathfrak{g}\subset \mathfrak{g}_{\mathbb{C}}$, we also have $N_{\mathcal{J}}(X,Y)=0$ for all $X,Y\in \mathfrak{g}$.  Hence, $\mathcal{J}: \mathfrak{g}\rightarrow \mathfrak{g}$ is the desired linear map that we seek.

We now construct a Samelson complex structure on $G_2$.  This amounts to computing a root space decomposition of $\mathfrak{g}_2$.  From Appendix \ref{AppendixG2Bracket}, it follows that a maximal abelian Lie subalgebra of $\mathfrak{g}_2$ is given by
\begin{equation}
\label{eqCartanG2}
\mathfrak{t}=\mbox{span}_{\mathbb{R}}\{b_1,b_{12}\}.
\end{equation}
One finds that the root space decomposition of $(\mathfrak{g}_2)_{\mathbb{C}}$ with respect to $\mathfrak{t}$ is given by
\begin{equation}
    \label{eqRootSpaceG2}
    (\mathfrak{g}_2)_{\mathbb{C}}=\mathfrak{t}_{\mathbb{C}}\oplus \bigoplus_{k=1}^6(\mathfrak{g}_{\alpha_k}\oplus \overline{\mathfrak{g}_{\alpha_k}}),
\end{equation}
where $\alpha_1,\alpha_2,\dots, \alpha_6$ are a system of postive roots and
\begin{equation}
    \label{eqGalpha1}
    \mathfrak{g}_{\alpha_1}=\mbox{span}_{\mathbb{C}}\{b_5+ib_6\},~\alpha_1(b_1)=i,~\alpha_1(b_{12})=i,
\end{equation}
\begin{equation}
    \label{eqGalpha2}
    \mathfrak{g}_{\alpha_2}=\mbox{span}_{\mathbb{C}}\{b_{14}+ib_{13}\},~\alpha_2(b_1)=i,~\alpha_2(b_{12})=-2i,
\end{equation}
\begin{equation}
    \label{eqGalpha3}
    \mathfrak{g}_{\alpha_3}=\mbox{span}_{\mathbb{C}}\{(2b_2+b_{14})+i(2b_4+b_{13})\},~\alpha_3(b_1)=i,~\alpha_3(b_{12})=0,
\end{equation}
\begin{equation}
    \label{eqGalpha4}
    \mathfrak{g}_{\alpha_4}=\mbox{span}_{\mathbb{C}}\{-(b_5+2b_{11})+i(b_6-2b_{10})\},~\alpha_4(b_1)=i,~\alpha_4(b_{12})=-i,
\end{equation}
\begin{equation}
    \label{eqGalpha5}
    \mathfrak{g}_{\alpha_5}=\mbox{span}_{\mathbb{C}}\{b_3+ib_7\},~\alpha_5(b_1)=2i,~\alpha_5(b_{12})=-i,
\end{equation}
\begin{equation}
    \label{eqGalpha6}
    \mathfrak{g}_{\alpha_6}=\mbox{span}_{\mathbb{C}}\{(b_7-2b_8)+i(b_3+2b_9)\},~\alpha_6(b_1)=0,~\alpha_6(b_{12})=-i.
\end{equation}
Let 
\begin{equation}
\label{eqEalpha1}
E_{\alpha_1}:=b_5+ib_6,
\end{equation}
\begin{equation}
\label{eqEalph2}
E_{\alpha_2}:=b_{14}+ib_{13},
\end{equation}
\begin{equation}
    \label{eqEalpha3}
    E_{\alpha_3}:=(2b_2+b_{14})+i(2b_4+b_{13}),
\end{equation}
\begin{equation}
    \label{eqEalpha4}
    E_{\alpha_4}:=-(b_5+2b_{11})+i(b_6-2b_{10}),
\end{equation}
\begin{equation}
    \label{eqEalpha5}
    E_{\alpha_5}:=b_3+ib_7,
\end{equation}
\begin{equation}
    \label{eqEalpha6}
    E_{\alpha_6}:=(b_7-2b_8)+i(b_3+2b_9).
\end{equation}
We define $\mathcal{J}':(\mathfrak{g}_2)_{\mathbb{C}}\rightarrow (\mathfrak{g}_2)_{\mathbb{C}}$ to be the linear map given by
\begin{align}
\label{eqJ'G2a}
\mathcal{J}'b_1:=-\frac{1}{\sqrt{3}}&b_1-\frac{2}{\sqrt{3}}b_{12},\hspace*{0.1in}\mathcal{J}'b_{12}:=\frac{2}{\sqrt{3}}b_1+\frac{1}{\sqrt{3}}b_{12},\\
\label{eqJ'G2b}
&\mathcal{J}'E_{\alpha_j}:=iE_{\alpha_j},~\mathcal{J}'\overline{E}_{\alpha_j}=-i\overline{E}_{\alpha_j}.
\end{align}
Then $(\mathcal{J}')^2=-\mbox{id}$, $N_{\mathcal{J'}}=0$, and $\mathcal{J'}\mathfrak{g}_2\subset \mathfrak{g}_2$.  The linear map on $\mathfrak{g}_2$ which induces the left-invariant integrable almost complex structure on $G_2$ is then $\mathcal{J}:=\mathcal{J'}\big|_{\mathfrak{g}_2}$.  Explicitly, $\mathcal{J}:\mathfrak{g}_2\rightarrow \mathfrak{g}_2$ is given by 
\begin{equation}
    \label{eqG2Jone}
    \mathcal{J}b_1=-\frac{1}{\sqrt{3}}b_1-\frac{2}{\sqrt{3}}b_{12},~\mathcal{J}b_2=-b_4,~\mathcal{J}b_3=-b_7,
\end{equation}
\begin{equation}
    \label{eqG2Jtwo}
    \mathcal{J}b_4=b_2,~\mathcal{J}b_5=-b_6,~\mathcal{J}b_6=b_5,
\end{equation}
\begin{equation}
    \label{eqG2Jthree}
    \mathcal{J}b_7=b_3,~\mathcal{J}b_8=b_3+b_9,~\mathcal{J}b_9=b_7-b_8,
\end{equation}
\begin{equation}
    \label{eqG2Jfour}
    \mathcal{J}b_{10}=b_5+b_{11},~\mathcal{J}b_{11}=b_6-b_{10},~\mathcal{J}b_{12}=\frac{2}{\sqrt{3}}b_1+\frac{1}{\sqrt{3}}b_{12},
\end{equation}
\begin{equation}
    \label{eqG2Jfive}
    \mathcal{J}b_{13}=b_{14},~\mathcal{J}b_{14}=-b_{13}.
\end{equation}
\begin{remark}
    \label{rmkKilling}
    In the Samelson construction, $\mathcal{J}|_{\mathfrak{t}}:\mathfrak{t}\rightarrow \mathfrak{t}$ is any linear map from $\mathfrak{t}$ to itself which squares to $-\mbox{id}|_{\mathfrak{t}}$.  As a consequence of this, it is always possible to construct a Samelson complex structure which is compatible with the Killing form. In the present case, $\mathfrak{t}$ is spanned by the basis elements $b_1,b_{12}$. Let $K$ be the  Killing form of $\mathfrak{g}_2$, which we recall is given by
    $$
        K(X,Y):=\mbox{tr}(\mbox{ad}_X\circ \mbox{ad}_Y),\hspace*{0.2in}\forall~X,Y\in \mathfrak{g}_2.
    $$
    By direct calculation, one verifies that (\ref{eqJ'G2a}) implies that 
    \begin{equation}
    \label{eqKJcompatCartan}
    K(\mathcal{J}H,\mathcal{J}H')=K(H,H'),\hspace*{0.2in}\forall~H,H'\in \mathfrak{t}.
    \end{equation}
    One can show\footnote{See Lemma \ref{lemSamelsonKilling}.} that (\ref{eqKJcompatCartan}) implies that $\mathcal{J}$ and $K$ must be compatible, that is, 
    $$
    K(\mathcal{J}X,\mathcal{J}Y)=K(X,Y),\hspace*{0.2in}\forall~X,Y\in\mathfrak{g}_2.
    $$
    This compatibility between $\mathcal{J}$ and $K$ turns out to be essential if one is to construct left-invariant $\mathcal{J}$-Hermtian metrics which are also SKT.

    Theorem 3.2 of \cite{FG2023} shows that if $G$ is a compact semi-simple Lie group, $I$ is a Samelson complex structure on $G$, and $g$ is a left-invariant $I$-Hermitian metric such that $(g,I)$ is an SKT structure, then $I$ must be compatible with some bi-invariant metric $g_0$ on $G$.  Consequently, if one is searching for left-invariant $I$-Hermitian metrics which are SKT, then Theorem 3.2 of \cite{FG2023} provides a necessary condition for the Samelson complex structure $I$.
    
    Returning to the present case, since $G_2$ is a compact simple Lie group, it follows that every bi-invariant metric on $G_2$ is equal to $-\lambda K$ (at the identity element) for some $\lambda>0$.  Since $\mathcal{J}$ is compatible with the Killing form, it follows that the left-invariant (integrable) almost complex structure $J$ induced by $\mathcal{J}$ is compatible with every bi-invariant metric on $G_2$.  As such, $\mathcal{J}$ satisfies the necessary condition on the Samelson complex structure given by Theorem 3.2 of \cite{FG2023}.  Ultimately, this is the justification for the form of (\ref{eqJ'G2a}).
\end{remark}

\section{SKT-metrics on $G_2$}
\label{secG2SKT}
\subsection{Review of SKT metrics}
In this section, we give a very brief review of SKT metrics.  A nice survey of SKT metrics is \cite{FT2009}.  An alternate viewpoint on SKT metrics which studies these objects within the framework of generalized geometry \cite{Hitchin2010} is given in  \cite{Cav2012}. 
\begin{definition}
    Let $(M,g,J)$ be a Hermitian manifold and let $\omega(\cdot,\cdot):=g(J\cdot,\cdot)$ be the associated fundamental form.  Then $(M,g,J)$ is SKT (or pluriclosed) if $\partial \overline{\partial}\omega = 0$.
\end{definition}
\noindent The following result shows that K\"{a}hler geometry is a special case of SKT geometry. 
\begin{proposition}
    Every K\"{a}hler manifold is SKT.  
\end{proposition}
\begin{proof}   
    Let $(M,g,J)$ be a K\"{a}hler manifold with fundamental form $\omega$.  Then 
    $$
        d\omega = \partial \omega + \overline{\partial}\omega = 0.
    $$
    Since $\omega$ is (1,1), $\partial \omega$ is (2,1) and $\overline{\partial}\omega$ is (1,2) which implies that 
    $$
        \partial \omega =0,\hspace*{0.2in} \overline{\partial}\omega = 0.
    $$
    From this, it immediately follows that $(M,g,J)$ is SKT.
\end{proof}
\noindent The following result will provide motivation for the ``T" in SKT.
\begin{proposition}
    \label{propSKTdc}
    Let $(M,g,J)$ be a Hermitian manifold with fundamental form $\omega$ and let $c(X,Y,Z):=d\omega(JX,JY,JZ)$.  Then $(M,g,J)$ is SKT if and only if $dc=0$.
\end{proposition}
\begin{proof}
    Let 
    $$
        \beta:=i\partial \omega - i\overline{\partial}\omega.
    $$
    Let $X^+$ and $X^-$ denote (1,0) and (0,1) vector fields.  Then $JX^+=iX^+$ and $JX^-=-iX^-$.
    Then
    \begin{align*}
        \beta(X^+,Y^+,Z^-)&=i(\partial\omega)(X^+,Y^+,Z^-)\\
        &=i(\partial\omega +\overline{\partial}\omega)(X^+,Y^+,Z^-)\\
        &=id\omega(X^+,Y^+,Z^-)\\
        &=d\omega(JX^+,JY^+,JZ^-).
    \end{align*}
    Likewise, one finds
    $$
      \beta(X^-,Y^-,Z^+)=d\omega(JX^-,JY^-,JZ^+).  
    $$
    Since $\beta$ and $c:=d\omega(J\cdot,J\cdot,J\cdot)$ are 3-forms of type $(2,1)+(1,2)$, the above calculation implies $\beta =c$.  From this, we have 
    \begin{align*}
        dc&=d\beta\\
        &=(\partial +\overline{\partial})\beta\\
        &=i\overline{\partial}\partial\omega -i\partial\overline{\partial}\omega\\
        &=-2i\partial\overline{\partial}\omega.
    \end{align*}
    Hence, $dc = 0$ if and only if $\partial\overline{\partial}\omega = 0$.
\end{proof}
Recall that a Hermitian connection on $(M,g,J)$ is a connection $\nabla$ satisfying $\nabla g =0$ and $\nabla J=0$.  There are infinitely many Hermitian connections on $(M,g,J)$ and each has its own unique torsion.  (More precisely, there is a one-to-one correspondence between the space of $TM$-valued (1,1)-forms $\Omega^{(1,1)}(M;TM)$ and the space of Hermitian connections on $(M,g,J)$ (cf Proposition 2.15 of \cite{PY2023})).  In particular, there is a unique Hermitian connection whose torsion $T^B$ is totally skew-symmetric in the sense that
$$
g(T^B(X,Y),Z)
$$
is a 3-form.  The Hermitian connection with torsion $T^B$ is the Bismut (or Strominger) connection. This connection has its roots in string theory (e.g. \cite{St1986,Hull1986,GHR1984}) but also arose for purely mathmatical reasons in \cite{Bismut1989}.   The torsion $T^B$ is given by none other than the 3-form $c$ appearing in Proposition \ref{propSKTdc}:
$$
g(T^B(X,Y),Z)=c:=d\omega(JX,JY,JZ).
$$
Hence, an SKT manifold is strictly non-K\"{a}hler if and only if its Bismut connection has nonzero torsion.  

Now let $G$ be a compact even dimensional semi-simple Lie group and let $g$ be any bi-invariant metric on $G$.  From \cite{SSTVP1988}, there always exists a Samelson complex structure $J$ on $G$ such that  $(g,J)$ is a (left-invariant) SKT structure.  To simplify things (and for the sake of making the paper as self contained as possible), we give a proof of this fact for the special case where $g$ is determined by a multiple of the Killing form, that is, $g_e = -\lambda K$ for $\lambda>0$. 
\begin{lemma}
    \label{lemSamelsonKilling}
    Let $G$ be an even dimensional compact semi-simple Lie group and let $K$ be the Killing form of $\mathfrak{g}:=\mbox{Lie}(G)$. Then there exists a left-invariant integrable almost complex structure $J$ which is compatible with $K$.
\end{lemma}
\begin{proof}
    Let $\mathfrak{t}$ be a maximal abelian Lie subalgebra of $\mathfrak{g}$ and let
    $$
        \mathfrak{g}_{\mathbb{C}}=\mathfrak{t}_\mathbb{C}\oplus \bigoplus_{\alpha\in \Delta^+}(\mathfrak{g}_{\alpha}\oplus \mathfrak{g}_{-\alpha})
    $$
    be the root space decomposition associated to $\mathfrak{t}$ where $\Delta^+$ is a system of positive roots. Recall that $K$ has the following properties with respect to the root space decomposition (cf \cite{Sep2007}): for all $\alpha,\beta\in \Delta^+\cup (-\Delta^+)\cup \{0\}$,
    $$
        K|_{\mathfrak{g}_{\alpha}\times \mathfrak{g}_{\beta}} = 0\hspace*{0.1in}\mbox{if}~\alpha+\beta \neq 0
    $$
    and 
    $$
    K|_{\mathfrak{g}_{\alpha}\times \mathfrak{g}_{-\alpha}}\hspace*{0.1in}\mbox{is non-degenerate},
    $$
    where $\mathfrak{g}_0=\mathfrak{t}_{\mathbb{C}}$.  Let $X_\alpha\in \mathfrak{g}_{\alpha}$, $Y_\beta\in \mathfrak{g}_{\beta}$, and $W\in \mathfrak{t}$ with $\alpha,\beta\in \Delta^+$.  Recall that by taking the complex conjugate, we have $\overline{X}_{\alpha}\in \mathfrak{g}_{-\alpha}$ and $\overline{Y}_\beta\in \mathfrak{g}_{-\beta}$. For the moment, let $J$ be any arbitrary Samelson complex structure. Since $\alpha+\beta\neq 0$ (due to the fact that $\alpha$ and $\beta$ are positive roots), we have
    $$
K(JX_{\alpha},JY_\beta)=i^2K(X_\alpha,Y_\beta)=0=K(X_\alpha,Y_\beta),
    $$
     $$
K(J\overline{X}_{\alpha},J\overline{Y}_\beta)=(-i)^2K(\overline{X}_\alpha,\overline{Y}_\beta)=0=K(\overline{X}_\alpha,\overline{Y}_\beta).
    $$
    Also,
    $$
    K(JX_\alpha,J\overline{Y}_{\beta})=(i)(-i)K(X_\alpha,\overline{Y}_{\beta})=K(X_\alpha,\overline{Y}_{\beta}),
    $$
    $$
    K(JX_\alpha,JW)=iK(X_\alpha,JW)=0=K(X_\alpha,W),
    $$
    $$
    K(J\overline{X}_{\alpha},JW)=-iK(\overline{X}_{\alpha},JW)=0=K(\overline{X}_{\alpha},W),
    $$
    where we have used the fact that $J\mathfrak{t}\subset \mathfrak{t}$ in the Samelson construction.  The only place where $J$ may not be compatible is on the restriction of $K$ to $\mathfrak{t}$.  
    
    In the Samelson construction, we are free to choose $J|_{\mathfrak{t}}$ to be any linear map which squares to $-\mbox{id}_{\mathfrak{t}}$.  Since $G$ is compact and semisimple, $K$ is negative definite (see e.g. \cite{Sep2007}).  Hence, we can choose a basis $W_1,\dots, W_{2k}$ of $\mathfrak{t}$ which is orthonormal with respect to $-K$.  Define $J|_{\mathfrak{t}}$ by
    $$
        JW_j:=W_{k+j},\hspace*{0.2in}JW_{k+j}:=-W_j,\hspace*{0.2in}\mbox{for }j\le k.
    $$
    With this definition, it now follows that $J$ is compatible on $K|_{\mathfrak{t}\times \mathfrak{t}}$.  Combining this fact with the previous calculation shows that $J$ is compatible with $K|_{\mathfrak{g}_{\mathbb{C}}\times \mathfrak{g}_{\mathbb{C}}}$.  In particular, since $J\mathfrak{g}\subset \mathfrak{g}$, $J$ is also compatible with $K|_{\mathfrak{g}\times \mathfrak{g}}$.  This completes the proof.
\end{proof}

\begin{proposition}
    \label{propSKTLieGroup}
    Let $G$ be a compact semi-simple Lie group with Killing form $K$ and let $J$ be any left-invariant integrable almost complex structure which is compatible with $K$.  Then for any real number $\lambda >0$, $(J,-\lambda K)$ is a left-invariant SKT structure on $G$ (where the bi-invariant metric determined by $-\lambda K$ is also denoted with the same symbol).
\end{proposition}
\begin{proof}
Let $\mathfrak{g}:=\mbox{Lie}(G)$.  Recall that the Killing form is compatible with the adjoint action of $G$, that is,
\begin{equation}
\label{eqKAd}
K(\mbox{Ad}_g\cdot, \mbox{Ad}_g\cdot)=K(\cdot,\cdot),\hspace*{0.2in}\forall~g\in G.
\end{equation}
This implies that $K$ is invariant with respect to the adjoint action of $\mathfrak{g}$ in the sense that
\begin{equation}
\label{eqKad}
K([X,Y],Z)+K(Y,[X,Z])=0,\hspace*{0.2in}\forall~X,Y,Z\in \mathfrak{g}.
\end{equation}
Since $G$ is compact and semi-simple, $K$ is negative definite.  Hence, $-\lambda K$ for $\lambda>0$ is positive definite.  Note that (\ref{eqKad}) is equivalent to the statement that the left-invariant metric determined by $-\lambda K$ is bi-invariant.  Let $K'$ denote the bi-invariant metric determined by $-\lambda K$.  Then $(J,K')$ is a (left-invariant) Hermitian structure on $G$.

We now turn to computing $c:=d\omega(J\cdot,J\cdot,J\cdot)$.  First, note that
\begin{align*}
    \omega([X,Y],Z)&=K'(J[X,Y],Z)\\
    &=-K'([X,Y],JZ)\\
    &=K'(Y,[X,JZ])\\
    &=\omega(Y,J[X,JZ]).
\end{align*}
Using this identity and the left-invariance of $\omega$ gives
\begin{align*}
    d\omega(X,Y,Z)&=-\omega([X,Y],Z)-\omega([Y,Z],X)-\omega([Z,X],Y)\\
    &=-\omega([X,Y],Z)-\omega(Z,J[Y,JX])+\omega([X,Z],Y)\\
    &=-\omega([X,Y],Z)-\omega(J[JX,Y],Z)+\omega(Z,J[X,JY])\\
    &=-\omega([X,Y],Z)-\omega(J[JX,Y],Z)-\omega(J[X,JY],Z)\\
    &=-\omega(N_J(X,Y),Z)-\omega([JX,JY],Z)\\
    &=-\omega([JX,JY],Z),
\end{align*}
where the Nijenhuis tensor $N_J=0$ since $J$ is assumed to be integrable. Then 
$$
c(X,Y,Z):=d\omega(JX,JY,JZ)=-\omega([X,Y],JZ)=-K'([X,Y],Z).
$$
So for $W,X,Y,Z\in \mathfrak{g}$, we have
\begin{align*}
    dc(W,X,Y,Z)&=-c([W,X],Y,Z)+c([W,Y],X,Z)-c([W,Z],X,Y)\\
    &-c([X,Y],W,Z)+c([X,Z],W,Y)-c([Y,Z],W,X)\\
    &=K'([[W,X],Y],Z)-K'([[W,Y],X],Z)+K'([[W,Z],X],Y)\\
    &+K'([[X,Y],W],Z)-K'([[X,Z],W],Y)+K'([[Y,Z],W],X)\\
    &=K'([[W,X],Y],Z)+K'([[Y,W],X],Z)+K'([[W,Z],X],Y)\\
    &+K'([[X,Y],W],Z)+K'([[Z,X],W],Y)+K'([[Y,Z],W],X)\\
    &=-K'([[X,W],Z],Y)+K'([[Y,Z],W],X)\\
    &=K'([Z,[X,W]],Y)+K'([[Y,Z],W],X)\\
    &=-K'([X,W],[Z,Y])+K'([[Y,Z],W],X)\\
    &=K'([W,X],[Z,Y])+K'([[Y,Z],W],X)\\
    &=-K'(X,[W,[Z,Y]])+K'([[Y,Z],W],X)\\
    &=-K'(X,[[Y,Z],W])+K'([[Y,Z],W],X)\\
    &=0,
\end{align*}
where the fourth equality follows from the Jacobi identity applied twice and the sixth and eighth equalities follow from the $\mbox{ad}$-invariance of $K$.  By Proposition \ref{propSKTdc}, $(J,K')=(J,-\lambda K)$ is a left-invariant SKT structure on $G$ (where the bi-invariant metric determined by $-\lambda K$ is also denoted with the same symbol).
\end{proof}
\noindent Proposition \ref{propSKTLieGroup} and Remark \ref{rmkKilling} now imply the following:
\begin{corollary}
    Let $\mathcal{J}:\mathfrak{g}_2\rightarrow \mathfrak{g}_2$ be the (integrable) almost complex structure given by (\ref{eqG2Jone})-(\ref{eqG2Jfive}) and let $K$ be the Killing form. Then $(\mathcal{J},-\lambda K)$ determines a left-invariant SKT structure on $G_2$ for all $\lambda>0$.
\end{corollary}

\subsection{A family of SKT-metrics on $G_2$}
Let $U$ be a complex vector space and let $\langle\cdot,\cdot\rangle$ be a Hermitian inner product on $U$.  Our convention for Hermitian inner products is to take $\mathbb{C}$-linearity in the \textit{second} argument so that 
$$
\langle\cdot,au_1+bu_2\rangle=a\langle\cdot,u_1\rangle+b\langle \cdot,u_2\rangle
$$
for $a,b\in \mathbb{C}$, $u_1,u_2\in U$.
Then conjugate symmetry implies 
$$
\langle au_1+bu_2,\cdot\rangle=\bar{a}\langle u_1,\cdot\rangle +\bar{b}\langle u_2,\cdot\rangle,
$$
where $\overline{a},~\overline{b}$ denotes the complex conjugate of $a,b$.  For the sake of completeness, we record the following elementary observation.
\begin{proposition}
    \label{propHermitianInner}
    Let $V$ be a real vector space and ${V}_{\mathbb{C}}:=V+iV$ the complexification of $V$.  Let $g$ be a positive definite inner product on $V$.  Then $\widehat{g}(u,v):=g(\overline{u},v)$ for $u,v\in V_{\mathbb{C}}$ is a Hermitian inner product on $V_{\mathbb{C}}$ (where $g(\overline{u},v)$ is evaluated by extending $g$ by $\mathbb{C}$-linearity in its arguments).  Conversely, if $\widehat{g}$ is any Hermitian inner product such that $\widehat{g}(u,v)\in \mathbb{R}$ for all $u,v\in V\subset V_{\mathbb{C}}$, then $\widehat{g}(u,v)=g(\overline{u},v)$ where $g:=\widehat{g}|_{V\times V}$.
\end{proposition}
\begin{proof}
    Suppose $g$ is a positive definite inner product on $V$ and $\widehat{g}(u,v):=g(\overline{u},v)$ for $u,v\in V_{\mathbb{C}}$.  Clearly $\widehat{g}$ is $\mathbb{C}$-linear in its second argument.  Now write $u=u_1+iu_2$, $v=v_1+iv_2$.  Then
    \begin{align*}
        \widehat{g}(u,v)&:=g(u_1-iu_2,v_1+iv_2)\\
            &=g(u_1,v_1)+(-i)(i)g(u_2,v_2)+ig(u_1,v_2)-ig(u_2,v_1)\\
            &=[g(u_1,v_1)+g(u_2,v_2)]+i[g(u_1,v_2)-g(u_2,v_1)].
    \end{align*}
    From this, we see that the symmetry of $g$ implies 
    $$
        \widehat{g}(v,u)=\overline{\widehat{g}(u,v)}
    $$
    and the positive definiteness of $g$ implies $\widehat{g}(u,u)>0$ for all nonzero $u\in V_{\mathbb{C}}$.

    On the other hand, if $\widehat{g}$ is a Hermitian inner product such that $\widehat{g}(u,v)\in \mathbb{R}$ for all $u,v\in V$, then $g:=\widehat{g}|_{V\times V}$ is clearly a positive definite inner product on $V$ and 
    \begin{align*}
        \widehat{g}(u,v)&=\widehat{g}(u_1+iu_2,v_1+iv_2)\\
        &=\widehat{g}(u_1,v_1)+\widehat{g}(u_1,iv_2)+\widehat{g}(iu_2,v_1)+\widehat{g}(iu_2,iv_2)\\
        &=\widehat{g}(u_1,v_1)+i\widehat{g}(u_1,v_2)-i\widehat{g}(u_2,v_1)+(-i)(i)\widehat{g}(u_2,v_2)\\
        &=[g(u_1,v_1)+{g}(u_2,v_2)]+i[{g}(u_1,v_2)-{g}(u_2,v_1)]\\
        &=g(\overline{u},v).
    \end{align*}
\end{proof}
\noindent Let $b_1,\dots, b_{14}$ be the basis of $\mathfrak{g}_2$ given by (\ref{eqG2basis1})-(\ref{eqG2basis5}) and let 
$$
E_{\alpha_1},\dots, E_{\alpha_6}
$$
be given by (\ref{eqEalpha1})-(\ref{eqEalpha6}). Consider the basis on $(\mathfrak{g}_2)_{\mathbb{C}}$ given by
\begin{equation}
    \label{eqComplexBasis}
    H_1,~H_2,~E_{\alpha_1},\dots, E_{\alpha_6},~\overline{E}_{\alpha_1},\dots, \overline{E}_{\alpha_6}
\end{equation}
where 
\begin{equation}
\label{eqH1H2}
H_1:=b_1,~H_2:=-\frac{1}{\sqrt{3}}b_1-\frac{2}{\sqrt{3}}b_{12}
\end{equation}
and let $\widehat{g}$ be the Hermitian inner product on $(\mathfrak{g}_2)_{\mathbb{C}}$ defined as follows: for $X,Y\in (\mathfrak{g}_2)_{\mathbb{C}}$ with
$$
X=X^1H_1+X^{2}H_{2}+\sum_{j=1}^6A^jE_{\alpha_j}+\sum_{j=1}^6B^j\overline{E}_{\alpha_j}
$$
$$
Y=Y^1H_1+Y^{2}H_{2}+\sum_{j=1}^6C^jE_{\alpha_j}+\sum_{j=1}^6D^j\overline{E}_{\alpha_j},
$$
\begin{align}
\label{eqG2ghat}
\widehat{g}(X,Y):=\lambda_0(\overline{X^1}Y^1+\overline{X^{2}}Y^{2})
+\sum_{j=1}^6\lambda_j(\overline{A^j}C^j+\overline{B^j}D^j)
\end{align}
where $\lambda_0,\lambda_1,\dots, \lambda_6\in \mathbb{R}$ are all positive.
\begin{proposition}
    \label{propG2HermitianInner}
    Let $\widehat{g}$ be defined by (\ref{eqG2ghat}).  The nonzero components of $\widehat{g}$ with respect to the basis $b_1,b_2,\dots, b_{14}$ is
    $$
        \widehat{g}(b_1,b_1)=\widehat{g}(b_{12},b_{12})=\lambda_0,~\widehat{g}(b_1,b_{12})=-\frac{1}{2}\lambda_0,
    $$
    $$
        \widehat{g}(b_2,b_2)=\frac{1}{8}(\lambda_2+\lambda_3),~\widehat{g}(b_2,b_{14})=-\frac{1}{4}\lambda_2,
    $$
    $$
    \widehat{g}(b_3,b_3)=\frac{1}{2}\lambda_5,~\widehat{g}(b_3,b_9)=-\frac{1}{4}\lambda_5,
    $$
    $$
    \widehat{g}(b_4,b_4)=\frac{1}{8}(\lambda_2+\lambda_3),~\widehat{g}(b_4,b_{13})=-\frac{1}{4}\lambda_2,
    $$
    $$
    \widehat{g}(b_5,b_5)=\frac{1}{2}\lambda_1,~\widehat{g}(b_5,b_{11})=-\frac{1}{4}\lambda_1,
    $$
    $$
    \widehat{g}(b_6,b_6)=\frac{1}{2}\lambda_1,~\widehat{g}(b_6,b_{10})=\frac{1}{4}\lambda_1,
    $$
    $$
    \widehat{g}(b_7,b_7)=\frac{1}{2}\lambda_5,~\widehat{g}(b_7,b_8)=\frac{1}{4}\lambda_5,
    $$
    $$
    \widehat{g}(b_8,b_8)=\widehat{g}(b_9,b_9)=\frac{1}{8}(\lambda_5+\lambda_6),
    $$
    $$
    \widehat{g}(b_{10},b_{10})=\widehat{g}(b_{11},b_{11})=\frac{1}{8}(\lambda_1+\lambda_4),
    $$
    $$
    \widehat{g}(b_{13},b_{13})=\widehat{g}(b_{14},b_{14})=\frac{1}{2}\lambda_2.
    $$
    In particular, $g:=\widehat{g}|_{\mathfrak{g}_2\times \mathfrak{g}_2}$ is a positive definite inner product on $\mathfrak{g}_2$ and $\widehat{g}(X,Y)=g(\overline{X},Y)$ for all $X,Y\in (\mathfrak{g}_2)_{\mathbb{C}}$.
\end{proposition}
\begin{proof}
    The basis $b_1,b_2,\dots, b_{14}$ expressed in terms of the basis 
    $$
        H_1,~H_{2},~E_{\alpha_1},\dots, E_{\alpha_6},~\overline{E}_{\alpha_1},\dots, \overline{E}_{\alpha_6}
    $$
    is given by 
    $$
        b_1=H_1,~\hspace*{0.1in}b_{12}=-\frac{1}{2}H_1-\frac{\sqrt{3}}{2}H_2,
    $$
    $$
        b_2=\frac{1}{4}(E_{\alpha_3}+\overline{E}_{\alpha_3})-\frac{1}{4}(E_{\alpha_2}+\overline{E}_{\alpha_2}),
    $$
    $$
        b_3=\frac{1}{2}(E_{\alpha_5}+\overline{E}_{\alpha_5}),
    $$
    $$
    b_4=-\frac{i}{4}(E_{\alpha_3}-\overline{E}_{\alpha_3})+\frac{i}{4}(E_{\alpha_2}-\overline{E}_{\alpha_2}),
    $$
    $$
    b_5=\frac{1}{2}(E_{\alpha_1}+\overline{E}_{\alpha_1}),
    $$
    $$
    b_6=-\frac{i}{2}(E_{\alpha_1}-\overline{E}_{\alpha_1}),
    $$
    $$
    b_7=-\frac{i}{2}(E_{\alpha_5}-\overline{E}_{\alpha_5}),
    $$
    $$
    b_8=-\frac{1}{4}(E_{\alpha_6}+\overline{E}_{\alpha_6})-\frac{i}{4}(E_{\alpha_5}-\overline{E}_{\alpha_5}),
    $$
    $$
     b_9=-\frac{1}{4}(E_{\alpha_5}+\overline{E}_{\alpha_5})-\frac{i}{4}(E_{\alpha_6}-\overline{E}_{\alpha_6}),
    $$
    $$
    b_{10}=\frac{i}{4}(E_{\alpha_4}-\overline{E}_{\alpha_4})-\frac{i}{4}(E_{\alpha_1}-\overline{E}_{\alpha_1}),
    $$
    $$
    b_{11}=-\frac{1}{4}(E_{\alpha_4}+\overline{E}_{\alpha_4})-\frac{1}{4}(E_{\alpha_1}+\overline{E}_{\alpha_1}),
    $$
    $$
    b_{13}=-\frac{i}{2}(E_{\alpha_2}-\overline{E}_{\alpha_2}),
    $$
    $$
    b_{14}=\frac{1}{2}(E_{\alpha_2}+\overline{E}_{\alpha_2}).
    $$
    The formula for the components of $\widehat{g}$ with respect to $b_1,\dots, b_{14}$ now follow from the above expressions and the definition of $\widehat{g}$. Since $\widehat{g}(b_i,b_j)\in \mathbb{R}$ for all $1\le i,j\le 14$, it follows that $\widehat{g}(X,Y)\in \mathbb{R}$ for all $X,Y\in \mathfrak{g}_2$.  The final statement now follows from Proposition \ref{propHermitianInner}.
\end{proof}
\begin{proposition}
    \label{propG2HermitianMetric}
    Let $\mathcal{J}:\mathfrak{g}_2\rightarrow \mathfrak{g}_2$ be the linear map defined by the Samelson construction given by (\ref{eqG2Jone})-(\ref{eqG2Jfive}) and let $g$ be the positive definite inner product on $\mathfrak{g}_2$ induced from the Hermitian inner product $\widehat{g}$ in  Proposition \ref{propG2HermitianInner}.  Then $g(\mathcal{J}\cdot,\mathcal{J}\cdot)=g(\cdot,\cdot)$.  In particular, $(g,\mathcal{J})$ determines a left-invariant Hermitian structure on $G_2$ (where the Hermitian metric induced by $g$ depends on seven positive parameters $\lambda_j>0$, $j=0,1,\dots, 6$).
\end{proposition}
\begin{proof}
    Let $\mathcal{J}': (\mathfrak{g}_2)_{\mathbb{C}}\rightarrow (\mathfrak{g}_2)_{\mathbb{C}}$ be the $\mathbb{C}$-linear extension of $\mathcal{J}$. To prove the statement, it suffices to check that 
    $$
        \widehat{g}(\mathcal{J}'\cdot,\mathcal{J}'\cdot)=\widehat{g}(\cdot,\cdot),
    $$
    on the basis (\ref{eqComplexBasis}) since $\widehat{g}(X,Y)=g(\overline{X},Y)$ for $X,Y\in (\mathfrak{g}_2)_{\mathbb{C}}$.  Note that 
    $$
        \mathcal{J}'H_1 = H_2,\hspace*{0.1in}\mathcal{J}'H_2 = -H_1.
    $$
    From the definition of $\widehat{g}$, we have
    $$
        \widehat{g}(\mathcal{J}'H_1,\mathcal{J}'H_{1})=\widehat{g}(H_{2},H_{2})=\lambda_0=\widehat{g}(H_1,H_1),
    $$
    $$
    \widehat{g}(\mathcal{J}'H_{2},\mathcal{J}'H_{2})=\widehat{g}(-H_1,-H_1)=\lambda_0=\widehat{g}(H_{2},H_{2}),
    $$
    $$
    \widehat{g}(\mathcal{J}'H_1,\mathcal{J}'H_{2})=\widehat{g}(H_{2},-H_1)=0=\widehat{g}(H_1,H_{2}).
    $$
    Since $\mathcal{J}$ is the Samelson construction associated to the root space (\ref{eqRootSpaceG2}), we have
    $$
        \mathcal{J}'E_{\alpha_j}=iE_{\alpha_j},~\mathcal{J}'\overline{E}_{\alpha_j}=-i\overline{E}_{\alpha_j}.
    $$
    Then 
    $$
     \widehat{g}(\mathcal{J}'H_1,\mathcal{J}'E_{\alpha_j})=i\widehat{g}(H_{2},E_{\alpha_j})=0=\widehat{g}(H_1,E_{\alpha_j}),
    $$
    $$
     \widehat{g}(\mathcal{J}'H_{2},\mathcal{J}'E_{\alpha_j})=-i\widehat{g}(H_{1},E_{\alpha_j})=0=\widehat{g}(H_{2},E_{\alpha_j}).
    $$
    Likewise, one finds
    $$
     \widehat{g}(\mathcal{J}'H_1,\mathcal{J}'\overline{E}_{\alpha_j})= \widehat{g}(H_1,\overline{E}_{\alpha_j}),
    $$
    $$
     \widehat{g}(\mathcal{J}'H_{2},\mathcal{J}'\overline{E}_{\alpha_j})= \widehat{g}(H_{2},\overline{E}_{\alpha_j}).
    $$
    Also,
    $$
        \widehat{g}(\mathcal{J}'E_{\alpha_j},\mathcal{J}'\overline{E}_{\alpha_k})=(-i)(-i)\widehat{g}(E_{\alpha_j},\overline{E}_{\alpha_k})=0=\widehat{g}(E_{\alpha_j},\overline{E}_{\alpha_k}),
    $$
    $$
    \widehat{g}(\mathcal{J}'E_{\alpha_j},\mathcal{J}'E_{\alpha_k})=(-i)(i)\widehat{g}(E_{\alpha_j},E_{\alpha_k})=\widehat{g}(E_{\alpha_j},E_{\alpha_k}),
    $$
    $$
    \widehat{g}(\mathcal{J}'\overline{E}_{\alpha_j},\mathcal{J}'\overline{E}_{\alpha_k})=(i)(-i)\widehat{g}(\overline{E}_{\alpha_j},\overline{E}_{\alpha_k})=\widehat{g}(\overline{E}_{\alpha_j},\overline{E}_{\alpha_k}).
    $$
\end{proof}
\begin{remark}
    For the rest of the paper, we will make no distinction between the left-invariant Hermitian structure on $G_2$ induced by the pair $(g,\mathcal{J})$ from Proposition \ref{propG2HermitianMetric} and the pair $(g,\mathcal{J})$ itself.  Hence, we will call $g$ the Hermitian metric and $\mathcal{J}$ the (integrable) almost complex structure.  Also, we will use the same symbol $\mathcal{J}$ to denote the $\mathbb{C}$-linear extension of $\mathcal{J}:\mathfrak{g}_2\rightarrow \mathfrak{g}_2$ as opposed to writing $\mathcal{J}'$ for the $\mathbb{C}$-linear extension.  Lastly, we identify $\wedge^k(\mathfrak{g}_2)_{\mathbb{C}}^\ast$ with the space of (complex-valued) left-invariant $k$-forms on $G_2$.
\end{remark}
 Let $(g,\mathcal{J})$ be the Hermitian structure from Proposition \ref{propG2HermitianMetric} and let
 $$
 \omega(\cdot,\cdot):=g(\mathcal{J}\cdot,\cdot)
 $$ 
 be the associated fundamental form.  Let
 \begin{equation}
     \label{eqComplexBasisDual}
     H_1^\ast,~H_2^\ast,~E^\ast_{1},\dots,E^\ast_{6},~\overline{E}^\ast_{1},\dots,\overline{E}^\ast_{6}
 \end{equation}
 be the dual basis of (\ref{eqComplexBasis}) (where we write $E^\ast_j$ and $\overline{E}_j^\ast$ instead of $E^\ast_{\alpha_j}$ and $\overline{E}_{\alpha_j}^\ast$ respectively to further simplify the notation).  Note that 
 $$
 \omega(\overline{X},Y) = g(\mathcal{J}\overline{X},Y)=g(\overline{\mathcal{J}X},Y)=\widehat{g}(
 \mathcal{J}X,Y). 
 $$
 The definition of $\widehat{g}$ then implies 
 \begin{equation}
     \label{eqOmegaComplexDualBasis}
     \omega = \lambda_0H_1^\ast\wedge H_2^\ast + i\sum_{j=1}^6 \lambda_jE_{j}^\ast\wedge\overline{E}_{j}^\ast.
 \end{equation}
 For $\beta\in \wedge^k(\mathfrak{g}_2)^\ast_{\mathbb{C}}$, let 
 $$
 \mathcal{J}\beta:= \beta(\mathcal{J}\cdot,\mathcal{J}\cdot,\cdots,\mathcal{J}\cdot).
 $$
Note that if $\beta=\theta_1\wedge \cdots \wedge\theta_k$ for $\theta_j\in (\mathfrak{g}_2)^\ast_{\mathbb{C}}$, then 
$$
\mathcal{J}\beta = (\mathcal{J}\theta_1)\wedge \cdots\wedge (\mathcal{J}\theta_k).
$$
We now focus on calculating the torsion of the Bismut connection, namely $c:=\mathcal{J}d\omega$ as well as its exterior derivative $dc$.  Note that 
\begin{align}
    \label{eqJHH}
    \mathcal{J}H_1^\ast &= -H_2^\ast,~\hspace*{0.2in}\mathcal{J}H_2^\ast = H_1^\ast,\\
    \label{eqJEEb}
    \mathcal{J}E^\ast_j &=iE^\ast_j,\hspace*{0.2in}\mathcal{J}\overline{E}^\ast_j=-i\overline{E}^\ast_j.
\end{align}

For $\theta\in (\mathfrak{g}_2)^\ast_{\mathbb{C}}$, we have
\begin{equation}
    \label{eqOneFormExterior}
    (d\theta)(X,Y)=-\theta([X,Y]),~\hspace*{0.1in}\forall~X,Y\in (\mathfrak{g}_2)_{\mathbb{C}}.
\end{equation}
To calculate $c$ and $dc$, we first calculate the exterior derivative of the basis elements (\ref{eqComplexBasisDual}).  Using (\ref{eqOneFormExterior}) and Appendix \ref{AppendixG2Complex}, one obtains the following:
\begin{align}
 \nonumber
dH_1^\ast &=-2iE_1^\ast\wedge\overline{E}_1^\ast-2iE_2^\ast\wedge \overline{E}^\ast_2-6iE_3^\ast \wedge\overline{E}_3^\ast \\
\label{eqdH1}
&-6iE_4^\ast\wedge \overline{E}^\ast_4-4iE^\ast_5\wedge \overline{E}^\ast_5,
 \end{align}
\begin{align}
    \nonumber
    dH_{2}^\ast&=2i\sqrt{3} E_1^\ast\wedge \overline{E}^\ast_1-2i\sqrt{3}E_2^\ast\wedge \overline{E}^\ast_2+2i\sqrt{3}  E^\ast_3\wedge \overline{E}^\ast_3\\
    \label{eqdH2}
    &-2i\sqrt{3}E_4^\ast\wedge \overline{E}^\ast_4-4i\sqrt{3}E^\ast_6\wedge \overline{E}^\ast_6,
\end{align}
\begin{align}
\label{eqdE1}
dE^\ast_{1}&=-iH_1^\ast\wedge E^\ast_{1}+i\sqrt{3}H_{2}^\ast\wedge E^\ast_{1}+6iE^\ast_{3}\wedge\overline{E}_{6}^\ast+2E^\ast_{5}\wedge\overline{E}_{2}^\ast,
\end{align}
\begin{align}
\label{eqdE2}
dE^\ast_{2}=-iH_1^\ast\wedge E^\ast_{2}-i\sqrt{3}H^\ast_{2}\wedge E^\ast_{2}+6iE^\ast_{4}\wedge E^\ast_{6}-2E^\ast_{5}\wedge \overline{E}^\ast_{1},
\end{align}
\begin{align}
    \nonumber
    dE^\ast_{3}&=-iH_1^\ast \wedge E^\ast_{3}+\frac{i}{\sqrt{3}}H_2^\ast\wedge E^\ast_3+2iE^\ast_{1}\wedge E^\ast_{6}\\
    \label{eqdE3}
    &-4iE^\ast_{4}\wedge\overline{E}^\ast_{6}-2E^\ast_{5}\wedge\overline{E}^\ast_{4},
\end{align}
\begin{align}
    \nonumber
    dE^\ast_{4}&=-iH_1^\ast\wedge E^\ast_{4}-\frac{i}{\sqrt{3}}H_{2}^\ast \wedge E^\ast_{4}+2iE^\ast_{2}\wedge \overline{E}^\ast_{6}\\
    \label{eqdE4}
    &-4iE^\ast_{3}\wedge E^\ast_{6}+2E^\ast_{5}\wedge \overline{E}^\ast_{3},
\end{align}
\begin{align}
    \label{eqdE5}
    dE^\ast_{5}=-2iH_1^\ast\wedge E^\ast_{5}-2E^\ast_{1}\wedge E^\ast_{2}+6E^\ast_{3}\wedge E^\ast_{4},
\end{align}
\begin{align}
    \label{eqdE6}
    dE^\ast_{6}=-\frac{2i}{\sqrt{3}}H^\ast_{2}\wedge E^\ast_{6}-2iE^\ast_{2}\wedge \overline{E}^\ast_{4}-2iE^\ast_{3}\wedge\overline{E}^\ast_{1}+4iE^\ast_{4}\wedge \overline{E}_{3}^\ast,
\end{align}
\begin{align}
\label{eqdE1bar}
    d\overline{E}^\ast_{1}=iH_1^\ast\wedge \overline{E}^\ast_{1}-i\sqrt{3}H_{2}^\ast \wedge \overline{E}^\ast_{1}-2E^\ast_{2}\wedge \overline{E}^\ast_{5}+6iE^\ast_{6}\wedge \overline{E}^\ast_{3},
\end{align}
\begin{align}
\label{eqdE2bar}
d\overline{E}^\ast_{2}=iH^\ast_1\wedge \overline{E}^\ast_{2}+i\sqrt{3}H^\ast_{2}\wedge \overline{E}^\ast_{2}+2E^\ast_{1}\wedge \overline{E}^\ast_{5}-6i\overline{E}^\ast_{4}\wedge \overline{E}^\ast_{6},
\end{align}
\begin{align}
    \nonumber
    d\overline{E}^\ast_{3}&=iH^\ast_1\wedge\overline{E}^\ast_{3}-\frac{i}{\sqrt{3}}H_2^\ast\wedge \overline{E}^\ast_3+2E^\ast_{4}\wedge \overline{E}^\ast_{5}\\
    \label{eqdE3bar}
    &-4iE^\ast_{6}\wedge \overline{E}^\ast_{4}-2i\overline{E}^\ast_{1}\wedge\overline{E}^\ast_{6},
\end{align}
\begin{align} 
    \nonumber
    d\overline{E}^\ast_{4}&=iH^\ast_1\wedge \overline{E}^\ast_{4}+\frac{i}{\sqrt{3}}H^\ast_2\wedge \overline{E}^\ast_{4}-2E^\ast_{3}\wedge \overline{E}^\ast_{5}\\
    \label{eqdE4bar}
    &+2iE^\ast_{6}\wedge \overline{E}^\ast_{2}+4i\overline{E}^\ast_{3}\wedge \overline{E}^\ast_{6},
\end{align}
\begin{align}
    \label{eqdE5bar}
    d\overline{E}^\ast_{5}=2iH_1^\ast\wedge\overline{E}^\ast_{5}-2\overline{E}^\ast_{1}\wedge \overline{E}^\ast_{2}+6\overline{E}^\ast_{3}\wedge \overline{E}^\ast_{4},
\end{align}
\begin{align}
    \label{eqdE6bar}
    d\overline{E}^\ast_{6}=\frac{2i}{\sqrt{3}}H^\ast_2\wedge \overline{E}^\ast_{6}-2iE^\ast_{1}\wedge \overline{E}^\ast_{3}+4iE^\ast_{3}\wedge \overline{E}^\ast_{4}-2iE^\ast_{4}\wedge \overline{E}^\ast_{2}.
\end{align}
Using (\ref{eqOmegaComplexDualBasis}), (\ref{eqJHH})-(\ref{eqJEEb}), and (\ref{eqdH1})-(\ref{eqdE6bar}), one obtains the expressions for $c:=\mathcal{J}d\omega$ and $dc$.  (This is a straightforward, but exceedingly lengthy calculation whose step by step details we omit.)
\allowdisplaybreaks
\begin{align}
\nonumber
c=&-2i\lambda_0 H_1^\ast\wedge E^\ast_1\wedge\overline{E}^\ast_1-2i\lambda_0 H_1^\ast \wedge E^\ast_2\wedge\overline{E}^\ast_2-6i\lambda_0H_1^\ast\wedge E^\ast_3\wedge \overline{E}^\ast_3\\
\nonumber
&-6i\lambda_0 H^\ast_1\wedge E^\ast_4\wedge \overline{E}^\ast_4-4i\lambda_0 H_1^\ast\wedge E^\ast_5\wedge \overline{E}^\ast_5+i2\sqrt{3}\lambda_0H_2^\ast\wedge E^\ast_1\wedge\overline{E}^\ast_1\\
\nonumber
&-i2\sqrt{3}\lambda_0 H_2^\ast\wedge E^\ast_2\wedge \overline{E}^\ast_2+i2\sqrt{3}\lambda_0 H^\ast_2\wedge E^\ast_3\wedge \overline{E}^\ast_3\\
\nonumber
&-i2\sqrt{3}\lambda_0 H_2^\ast\wedge E^\ast_4\wedge \overline{E}^\ast_4-i4\sqrt{3}\lambda_0 H_2^\ast\wedge E^\ast_6\wedge \overline{E}^\ast_6\\
\nonumber
&+(-6i\lambda_1+2i\lambda_3-2i\lambda_6)E^\ast_3\wedge\overline{E}^\ast_1\wedge\overline{E}^\ast_6\\
\nonumber
&+(-2\lambda_1-2\lambda_2+2\lambda_5)E^\ast_5\wedge \overline{E}^\ast_1\wedge \overline{E}^\ast_2\\
\nonumber
&+(-2\lambda_1-2\lambda_2+2\lambda_5)E^\ast_1\wedge E^\ast_2\wedge \overline{E}^\ast_5\\
\nonumber
&+(6i\lambda_1-2i\lambda_3+2i\lambda_6)E^\ast_1\wedge E^\ast_6\wedge \overline{E}^\ast_3\\
\nonumber
&+(-6i\lambda_2+2i\lambda_4+2i\lambda_6)E^\ast_4\wedge E^\ast_6\wedge \overline{E}^\ast_2\\
\nonumber
&+(6i\lambda_2-2i\lambda_4-2i\lambda_6)E^\ast_2\wedge\overline{E}^\ast_4\wedge\overline{E}^\ast_6\\
\nonumber
&+(4i\lambda_3-4i\lambda_4+4i\lambda_6)E^\ast_4\wedge \overline{E}^\ast_3\wedge \overline{E}^\ast_6\\
\nonumber
&+(2\lambda_3+2\lambda_4-6\lambda_5)E^\ast_5\wedge\overline{E}^\ast_3\wedge\overline{E}^\ast_4\\
\nonumber
&+(2\lambda_3+2\lambda_4-6\lambda_5)E^\ast_3\wedge E^\ast_4\wedge \overline{E}^\ast_5\\
\label{eqG2cExpression}
&+(-4i\lambda_3+4i\lambda_4-4i\lambda_6)E^\ast_3\wedge E^\ast_6\wedge\overline{E}^\ast_4
\end{align}
\allowdisplaybreaks
\begin{align}
    \nonumber
dc&=(-16\lambda_0+8\lambda_1+8\lambda_2-8\lambda_5)E^\ast_1\wedge E^\ast_2\wedge \overline{E}^\ast_1\wedge \overline{E}^\ast_2\\
\nonumber
&+(48\lambda_0-24\lambda_1+8\lambda_3-8\lambda_6)E^\ast_1\wedge E^\ast_3\wedge \overline{E}^\ast_1\wedge \overline{E}^\ast_3\\
\nonumber
&+(16\lambda_0-8\lambda_1-8\lambda_2+8\lambda_5)E^\ast_1\wedge E^\ast_5\wedge \overline{E}^\ast_1\wedge \overline{E}^\ast_5\\
\nonumber
&+(48\lambda_0+24\lambda_2-8\lambda_4-8\lambda_6)E^\ast_2\wedge E^\ast_4\wedge \overline{E}^\ast_2\wedge \overline{E}^\ast_4\\
\nonumber
&+(16\lambda_0-8\lambda_1-8\lambda_2+8\lambda_5)E^\ast_2\wedge E^\ast_5\wedge \overline{E}^\ast_2\wedge \overline{E}^\ast_5\\
\nonumber
&+(48\lambda_0-8\lambda_3+56\lambda_4-72\lambda_5-32\lambda_6)E^\ast_3\wedge E^\ast_4\wedge\overline{E}^\ast_3\wedge \overline{E}^\ast_{4}\\
\nonumber
&+(48\lambda_0-8\lambda_3-8\lambda_4+24\lambda_5)E^\ast_3\wedge E^\ast_5\wedge \overline{E}^\ast_3\wedge \overline{E}^\ast_5\\
\nonumber
&+(48\lambda_0-8\lambda_3-8\lambda_4+24\lambda_5)E^\ast_4\wedge E^\ast_5\wedge \overline{E}^\ast_4\wedge \overline{E}^\ast_5\\
\nonumber
&+(-48\lambda_0+24\lambda_1-8\lambda_3+8\lambda_6)E^\ast_1\wedge E^\ast_6\wedge \overline{E}^\ast_1\wedge \overline{E}^\ast_6\\
\nonumber
&+(48\lambda_0+24\lambda_2-8\lambda_4-8\lambda_6)E^\ast_2\wedge E^\ast_6\wedge \overline{E}^\ast_2\wedge \overline{E}^\ast_6\\
\nonumber
&+(-48\lambda_0-72\lambda_1+56\lambda_3-32\lambda_4+8\lambda_6)E^\ast_3\wedge E^\ast_6\wedge \overline{E}^\ast_3\wedge \overline{E}^\ast_6\\
\nonumber
&+(48\lambda_0-72\lambda_2-32\lambda_3+56\lambda_4-8\lambda_6)E^\ast_4\wedge E^\ast_6\wedge \overline{E}^\ast_4\wedge \overline{E}^\ast_6\\
\nonumber
&+(24i\lambda_1+24i\lambda_2-8i\lambda_3-8i\lambda_4)E^\ast_2\wedge E^\ast_3\wedge \overline{E}^\ast_5\wedge \overline{E}^\ast_6\\
\nonumber
&+(-24i\lambda_1-24i\lambda_2+8i\lambda_3+8i\lambda_4)E^\ast_5\wedge E^\ast_6\wedge \overline{E}^\ast_2\wedge \overline{E}^\ast_3\\
\nonumber
&+(-24\lambda_2-8\lambda_3+24\lambda_5+8\lambda_6)E^\ast_3\wedge E^\ast_4\wedge \overline{E}^\ast_1\wedge \overline{E}^\ast_2\\
\label{eqG2dc}
&+(-24\lambda_2-8\lambda_3+24\lambda_5+8\lambda_6)E^\ast_1\wedge E^\ast_2\wedge \overline{E}^\ast_3\wedge \overline{E}^\ast_4
\end{align}
\begin{theorem}
    \label{thmG2SKTmetrics}
    Let $\mathcal{J}:\mathfrak{g}_2\rightarrow \mathfrak{g}_2$ be the Samelson complex structure given by (\ref{eqG2Jone})-(\ref{eqG2Jfive}) and let $g$ be the symmetric bilinear form on $\mathfrak{g}_2$ whose nonzero components with respect to the basis $b_1,\dots, b_{14}$ are given by
    \begin{align}
        \label{eqSKTmetric1}
        &g(b_1,b_1)=g(b_{12},b_{12})=\frac{2a_1-3a_2+a_3}{12},\\
        \label{eqSKTmetric2}
        &g(b_1,b_{12})=-\frac{2a_1-3a_2+a_3}{24},~g(b_2,b_2)=\frac{2a_1+3a_2+a_3}{24},\\
        \label{eqSKTmetric3}
        &g(b_2,b_{14})=-\frac{-a_1+3a_2+a_3}{12},~g(b_3,b_3)=\frac{1}{2}a_2,~g(b_3,b_9)=-\frac{1}{4}a_2,\\
        \label{eqSKTmetric4}
        &g(b_4,b_4)=\frac{2a_1+3a_2+a_3}{24},~g(b_4,b_{13})=-\frac{-a_1+3a_2+a_3}{12}\\
        \label{eqSKTmetric5}
        &g(b_5,b_5)=\frac{4a_1-3a_2-a_3}{12},~g(b_5,b_{11})=-\frac{4a_1-3a_2-a_3}{24}\\
        \label{eqSKTmetric6}
        &g(b_6,b_6)=\frac{4a_1-3a_2-a_3}{12},~g(b_6,b_{10})=\frac{4a_1-3a_2-a_3}{24},\\
        \label{eqSKTmetric7}
        &g(b_7,b_7)=\frac{1}{2}a_2,~g(b_7,b_8)=\frac{1}{4}a_2,~g(b_8,b_8)=g(b_9,b_9)=\frac{1}{8}(a_2+a_3),\\
        \label{eqSKTmetric8}
        &g(b_{10},b_{10})=g(b_{11},b_{11})=\frac{2a_1+3a_2+a_3}{24},\\
        \label{eqSKTmetric9}
        &g(b_{13},b_{13})=g(b_{14},b_{14})=\frac{-a_1+3a_2+a_3}{6},
    \end{align}
    where the parameters $a_1,a_2,a_3$ satisfy the following inequalities
    \begin{equation}
        \label{eqParameterInequalities}
        0<a_2<a_1,\hspace*{0.2in}\gamma<a_3<4a_1-3a_2,
    \end{equation}
    with $\gamma:=\mbox{max}\{3a_2-2a_1,~a_1-3a_2,~0\}$.  Then 
    \begin{itemize}
        \item[(1)] $(g,\mathcal{J})$ determines a left-invariant SKT structure on $G_2$;
        \item[(2)] $g$ is invariant under the right action of the maximal torus $T$ whose Lie algebra is $\mathfrak{t}:=\mbox{span}_{\mathbb{R}}\{b_1,~b_{12}\}$; and
        \item[(3)] $g=-\lambda K$, (where $K$ is the Killing form and $\lambda>0$) if and only if
    $$
     a_1=96\lambda,~a_2=32\lambda,~a_3=96\lambda.
    $$
    \end{itemize}
    Conversely, if $g$ is any $\mathcal{J}$-Hermitian metric satisfying (1) and (2), then $g$ must be of the form given by (\ref{eqSKTmetric1})-(\ref{eqSKTmetric9}).
\end{theorem}
\begin{proof}
    For (1), let $g(\cdot,\cdot):=\omega(\cdot,\mathcal{J}\cdot)$ where $\omega$ is given by (\ref{eqOmegaComplexDualBasis}) with $\lambda_0,\dots, \lambda_6$ all positive.  Then $g$ is precisely the $\mathcal{J}$-Hermitian metric of Proposition \ref{propG2HermitianMetric} which is parameterized by $\lambda_j>0$, $j=0,1,\dots, 6$. By Proposition \ref{propSKTdc}, $(g,\mathcal{J})$ is an SKT structure precisely if $dc=0$.  Setting all the coefficients of $dc$ in (\ref{eqG2dc}) to zero and solving the resulting linear system, one obtains the following general solution:
    \begin{align}
        \label{eqLambda0}
        \lambda_0 &= \frac{1}{6}\lambda_3-\frac{1}{4}\lambda_5+\frac{1}{12}\lambda_6\\
        \label{eqLambda1}
        \lambda_1 &=\frac{2}{3}\lambda_3-\frac{1}{2}\lambda_5-\frac{1}{6}\lambda_6\\
        \label{eqLambda2}
        \lambda_2 &= -\frac{1}{3}\lambda_3 +\lambda_5+\frac{1}{3}\lambda_6\\
        \label{eqLambda4}
        \lambda_4 &= \frac{3}{2}\lambda_5+\frac{1}{2}\lambda_6,
    \end{align}
    where $\lambda_3$, $\lambda_5$, and $\lambda_6$ are free.  However, for $g$ to be positive definite, we also require $\lambda_j>0$ for $j=0,1,\dots, 6$.  Consequently, some additional constraints must be placed on $\lambda_3$, $\lambda_5$, and $\lambda_6$.  With $\lambda_0$, $\lambda_1$, $\lambda_2$, and $\lambda_4$ given by (\ref{eqLambda0})-(\ref{eqLambda4}), one finds that $\lambda_j>0$ for $j=0,1,\dots, 6$ if and only if $\lambda_3$, $\lambda_5$, and $\lambda_6$ satisfy the following inequalities:
    \begin{align*}
        0<\lambda_5<\lambda_3,\hspace*{0.2in}\gamma<\lambda_6<4\lambda_3-3\lambda_5,
    \end{align*}
    where 
    $$
        \gamma:=\mbox{max}\{3\lambda_5-2\lambda_3,~\lambda_3-3\lambda_5,~0\}.
    $$
    Substituting $\lambda_3=a_1$, $\lambda_5=a_2$, $\lambda_6=a_3$, and (\ref{eqLambda0})-(\ref{eqLambda4}) into the component formulas for the metric in Proposition \ref{propG2HermitianInner} gives (\ref{eqSKTmetric1})-(\ref{eqSKTmetric9}).

    For (2), let $T$ be the maximal torus of $G_2$ whose Lie algebra is $\mathfrak{t}:=\mbox{span}_{\mathbb{R}}\{b_1,~b_{12}\}$.  To say that the left-invariant metric determined by $g$ is also invariant with respect to the right action by $T$ is equivalent to the statement that $g$ is $\mbox{ad}$-invariant with respect to $\mathfrak{t}$.  By direct calculation, one verifies that 
    \begin{align}
        \label{eqMaximalTorusb1}
        g([b_1,b_j],b_k)+g(b_j,[b_1,b_k])&=0\\
        \label{eqMaximalTorusb12}
        g([b_{12},b_j],b_k)+g(b_j,[b_{12},b_k])&=0
    \end{align}
    for all $1\le j,k\le 14$.
    
    For (3), consider the basis on $(\mathfrak{g}_2)_{\mathbb{C}}$ given by (\ref{eqComplexBasis}). It follows from (\ref{eqOmegaComplexDualBasis}) that 
    $$
    g(H_1,H_1) = g(H_2,H_2)= \lambda_0,\hspace*{0.2in}g(E_{\alpha_j},\overline{E}_{\alpha_j})=\lambda_j
    $$
    and all other pairings of the basis elements of (\ref{eqComplexBasis}) are zero.

    By direct calculation of the Killing form, one finds 
    $$
    K(H_1,H_1)=K(H_2,H_2)=-16,
    $$
    $$
    K(E_{\alpha_1},\overline{E}_{\alpha_1})=K(E_{\alpha_2},\overline{E}_{\alpha_2})=K(E_{\alpha_5},\overline{E}_{\alpha_5})=-32,
    $$
    $$
    K(E_{\alpha_3},\overline{E}_{\alpha_3})=K(E_{\alpha_4},\overline{E}_{\alpha_4})=K(E_{\alpha_6},\overline{E}_{\alpha_6})=-96,
    $$
    with all other pairings of the basis elements of (\ref{eqComplexBasis}) zero.  Hence, $g=-\lambda K$ for $\lambda>0$ if and only if
    \begin{equation}
        \label{eqLambdaKilling1}
        \lambda_0=16\lambda,\hspace*{0.1in}\lambda_1=\lambda_2=\lambda_5=32\lambda,~\lambda_3=\lambda_4=\lambda_6=96\lambda.
    \end{equation}
    Since the derivation of (\ref{eqSKTmetric1})-(\ref{eqSKTmetric9}) is obtained by substituting $\lambda_3=a_1$, $\lambda_5=a_2$, $\lambda_6=a_3$, and (\ref{eqLambda0})-(\ref{eqLambda4}) into the component formulas of the metric in Proposition \ref{propG2HermitianInner}, it follows that  we must have
    \begin{equation}
        \label{eqa1a2a3Killing}
        a_1=96\lambda,\hspace*{0.1in}a_2=32\lambda,\hspace*{0.1in}a_3=96\lambda.
    \end{equation}
    One now verifies that with the above values for $\lambda_3=a_1$, $\lambda_5=a_2$, and $\lambda_6=a_3$, the values of $\lambda_0$, $\lambda_1$, $\lambda_2$, and $\lambda_4$ given by (\ref{eqLambda0})-(\ref{eqLambda4}) coincide exactly with the corresponding values appearing in (\ref{eqLambdaKilling1}). This proves that the SKT metric given by (\ref{eqSKTmetric1})-(\ref{eqSKTmetric9}) satisfies $g=-\lambda K$ precisely when $a_1$, $a_2$, and $a_3$ are given by (\ref{eqa1a2a3Killing}).

    The very last statement follows from Proposition \ref{propSKTtotal} below.
\end{proof}

\begin{remark}
     As a practical matter, one should always check a formula against actual examples.  This is especially true when the derivation of that formula required lengthy calculations.  For this reason, the SKT metric of Theorem \ref{thmG2SKTmetrics} was checked against numerous examples.  Specifically, a program was written (in the language \textit{Octave}) that would randomly generate the three parameters $a_1,a_2,a_3$ subject to the inequalities given in Theorem \ref{thmG2SKTmetrics} and then test that the metric given by (\ref{eqSKTmetric1})-(\ref{eqSKTmetric9}) satisfied all the expected conditions.  In each example, the resulting metric $g$ was found to be positive definite, invariant under $\mathcal{J}$, and the pair $(g,\mathcal{J})$ satisfied $dc=0$, that is, $(g,\mathcal{J})$ determines a left-invariant SKT structure on $G_2$ (as expected).
\end{remark}
\begin{proposition}
    \label{propSKTtotal}
    Let $T$ be the maximal torus of $G_2$ whose Lie algebra is $\mathfrak{t}:=\mbox{span}_{\mathbb{R}}\{b_1,~b_{12}\}$.  If $h$ is a $\mathcal{J}$-Hermitian metric such that $h$ is invariant under the right action of $T$ and for which $(h,\mathcal{J})$ is SKT, then $h$ is a member of the 3-parameter family of left-invariant SKT metrics on $G_2$ given in Theorem \ref{thmG2SKTmetrics}.
\end{proposition}
\begin{proof}
    From (\ref{eqGalpha1})-(\ref{eqGalpha6}), we have 
    \begin{align}
        \label{eqAlpha1H}
        &\alpha_1(H_1)=i,~\alpha_1(H_2)=-\sqrt{3}i,\\
        \label{eqAlpha2H}
        &\alpha_2(H_1)=i,~\alpha_2(H_2)=\sqrt{3}i,\\
        \label{eqAlpha3H}
        &\alpha_3(H_1)=i,~\alpha_3(H_2)=-\frac{1}{\sqrt{3}}i,\\
        \label{eqAlpha4H}
        &\alpha_4(H_1)=i,~\alpha_4(H_2)=\frac{1}{\sqrt{3}}i,\\
        \label{eqAlpha5H}
        &\alpha_5(H_1)=2i,~\alpha_5(H_2)=0,\\
        \label{eqAlpha6H}
        &\alpha_6(H_1)=0,~\alpha_6(H_2)=\frac{2}{\sqrt{3}}i.
    \end{align}
    Now let $h$ be a left-invariant $\mathcal{J}$-Hermitian metric which is invariant under the right action of $T$.  Since $\mathcal{J}H_1=H_2$, $\mathcal{J}$-invariance of $h$ implies 
    $$
    h(H_1,H_1)=h(H_2,H_2),\hspace*{0.1in}h(H_1,H_2)=0.
    $$
    Also for all $H\in \mathfrak{t}$, $X,Y\in (\mathfrak{g}_2)_{\mathbb{C}}$, invariance under the right action of $T$ implies 
    \begin{equation}
        \label{eqRightTh}
        h([H,X],Y)+h(X,[H,Y])=0.
    \end{equation}
    Consider the case where $X=H'\in \mathfrak{t}$ and $Y=E_{\alpha_j}$.  Then (\ref{eqRightTh}) reduces to
    $$
        \alpha_j(H)h(H',E_{\alpha_j})=0.
    $$
    Since $\alpha_j(H)\neq 0$ for some $H\in \mathfrak{t}$, it follows that 
    $$
    h(H',E_{\alpha_j})=0,\hspace*{0.1in}\forall~H'\in \mathfrak{t},~\forall~j=1,\dots, 6.
    $$
    Repeating the argument with $\overline{E}_{\alpha_j}$, one simply replaces $\alpha_j$ with $-\alpha_j$.  Consequently, one also finds
    $$
    h(H',\overline{E}_{\alpha_j})=0,\hspace*{0.1in}\forall~H'\in \mathfrak{t},~\forall~j=1,\dots, 6.
    $$
    By $\mathcal{J}$-invariance of $h$, we have
    $$
    h(E_{\alpha_j},E_{\alpha_k})=h(\mathcal{J}E_{\alpha_j}, \mathcal{J}E_{\alpha_k})=(i)^2h(E_{\alpha_j},E_{\alpha_k})=-h(E_{\alpha_j},E_{\alpha_k}).
    $$
    Hence,
    $$
        h(E_{\alpha_j},E_{\alpha_k})=0,\hspace*{0.1in}\forall~j,k
    $$
    Likewise, one also has
    $$
        h(\overline{E}_{\alpha_j},\overline{E}_{\alpha_k})=0,\hspace*{0.1in}\forall~j,k
    $$
    Lastly, consider the case where $X=E_{\alpha_j}$ and $Y=\overline{E}_{\alpha_k}$ in (\ref{eqRightTh}).  This gives
    $$
    \left(\alpha_j(H)-\alpha_k(H)\right)h(E_{\alpha_j},\overline{E}_{\alpha_k})=0.
    $$
    This implies
    $$
    h(E_{\alpha_j},\overline{E}_{\alpha_k})=0,\hspace*{0.1in}\forall~j\neq k. 
    $$
    Let $\mu_0:=h(H_1,H_1)=h(H_2,H_2)>0$ and $\mu_j:=h(E_{\alpha_j},\overline{E}_{\alpha_j})>0$ for $j=1,\dots,6$.  Then the fundamental form associated to $h$ is 
    \begin{equation}
    \label{eqOmegah}
    \omega_h=\mu_0H_1^\ast \wedge H_2^\ast + i\sum_{j=1}^6\mu_j E^\ast_j\wedge \overline{E}^\ast_j.
    \end{equation}
    Recall that the 7-parameter family of $\mathcal{J}$-Hermitian metrics in Proposition \ref{propG2HermitianMetric} consists of all $\mathcal{J}$-Hermitian metrics whose fundamental form is of the form (\ref{eqOmegaComplexDualBasis}). Comparing $\omega_h$ to (\ref{eqOmegaComplexDualBasis}), we see that $h$ is a member of the 7-parameter family of $\mathcal{J}$-Hermitian metrics given in Proposition \ref{propG2HermitianMetric}. 
    
    In the proof of Theorem \ref{thmG2SKTmetrics}, all SKT metrics were extracted from the 7-parameter family of $\mathcal{J}$-Hermitian metrics given in Proposition \ref{propG2HermitianMetric}.  This  resulted in a 3-parameter family of $\mathcal{J}$-Hermitian SKT metrics which were also invariant under the right action of $T$.  Consequently, if $h$ is also SKT, it must be a member of the 3-parameter family of SKT metrics in Theorem \ref{thmG2SKTmetrics}.
\end{proof}
    
\begin{remark}
    It would be interesting to study the curvature of the Bismut connections associated to the 3-parameter family of (left-invariant) SKT metrics on $G_2$ given in Theorem \ref{thmG2SKTmetrics}.  As this will most likely require a good bit of calculation, we leave this particular problem for future study.
\end{remark}
\appendix
\section{Nonzero bracket relations on $\mathfrak{g}_2$}
\label{AppendixG2Bracket}
Let $b_1,b_2,\dots, b_{14}$ be the basis of $\mathfrak{g}_2$ given by (\ref{eqG2basis1})-(\ref{eqG2basis5}).  The nonzero bracket relations are
$$
[b_1,b_2]=-b_4,~[b_1,b_3]=-2b_7,~[b_1,b_4]=b_2,
$$
$$
[b_1,b_5]=-b_6,~[b_1,b_6]=b_5,~[b_1,b_7]=2b_3,
$$
$$
[b_1,b_8]=b_3,~[b_1,b_9]=b_7,~[b_1,b_{10}]=b_5+b_{11},
$$
$$
[b_1,b_{11}]=b_6-b_{10},~[b_1,b_{13}]=b_{14},~[b_1,b_{14}]=-b_{13},
$$
$$
[b_2,b_3]=-b_5-b_{11},~[b_2,b_4]=-b_1,~[b_2,b_5]=-b_9,
$$
$$
[b_2,b_6]=b_8,~[b_2,b_7]=-b_{10},~[b_2,b_8]=-2b_{10},
$$
$$
[b_2,b_9]=b_5,~[b_2,b_{10}]=2b_8,~[b_2,b_{11}]=b_3+b_9,~
$$
$$
[b_2,b_{12}]=b_{13},~[b_2,b_{13}]=-b_{12},~[b_3,b_4]=-b_{10},
$$
$$
[b_3,b_5]=b_{14},~[b_3,b_6]=-b_{13},~[b_3,b_7]=-2b_1,
$$
$$
[b_3,b_8]=-b_1,~[b_3,b_{10}]=b_4,~[b_3,b_{11}]=-b_2-b_{14},
$$
$$
[b_3,b_{12}]=-b_7,~[b_3,b_{13}]=b_6,~[b_3,b_{14}]=-b_5,
$$
$$
[b_4,b_5]=b_7-b_8,~[b_4,b_6]=-b_3-b_9,~[b_4,b_7]=-b_5-b_{11},
$$
$$
[b_4,b_8]=b_5+b_{11},~[b_4,b_9]=b_6-b_{10},~[b_4,b_{10}]=-b_3,
$$
$$
[b_4,b_{11}]=b_7-b_8,~[b_4,b_{12}]=-b_{14},~[b_4,b_{14}]=b_{12},
$$
$$
[b_5,b_6]=-2b_1-2b_{12},~[b_5,b_7]=-b_{13},~[b_5,b_8]=-b_4-b_{13},
$$
$$
[b_5,b_9]=-b_2,~[b_5,b_{10}]=-b_1-b_{12},~[b_5,b_{12}]=b_6,
$$
$$
[b_5,b_{13}]=b_7,~[b_5,b_{14}]=b_3,~[b_6,b_7]=-b_{14},
$$
$$
[b_6,b_8]=b_2,~[b_6,b_9]=-b_4-b_{13},~[b_6,b_{11}]=-b_1-b_{12},
$$
$$
[b_6,b_{12}]=-b_5,~[b_6,b_{13}]=-b_3,~[b_6,b_{14}]=b_7,
$$
$$
[b_7,b_9]=-b_1,~[b_7,b_{10}]=-b_2,~[b_7,b_{11}]=-b_4-b_{13},
$$
$$
[b_7,b_{12}]=b_3,~[b_7,b_{13}]=-b_5,~[b_7,b_{14}]=-b_6,
$$
$$
[b_8,b_9]=-b_1-b_{12},~[b_8,b_{10}]=-2b_2,~[b_8,b_{11}]=-b_{13},
$$
$$
[b_8,b_{12}]=b_3+b_9,~[b_8,b_{13}]=b_{11},~[b_8,b_{14}]=-b_{10},
$$
$$
[b_9,b_{10}]=-b_4-b_{13},~[b_9,b_{11}]=-b_2-b_{14},~[b_9,b_{12}]=b_7-b_8,
$$
$$
[b_9,b_{13}]=-b_6+b_{10},~[b_9,b_{14}]=b_5+b_{11},~[b_{10},b_{11}]=-b_{12},
$$
$$
[b_{10},b_{12}]=b_{11},~[b_{10},b_{13}]=-b_3-b_9,~[b_{10},b_{14}]=b_8,
$$
$$
[b_{11},b_{12}]=-b_{10},~[b_{11},b_{13}]=-b_8,~[b_{11},b_{14}]=-b_3-b_9,
$$
$$
[b_{12},b_{13}]=-2b_{14},~[b_{12},b_{14}]=2b_{13},~[b_{13},b_{14}]=-2b_{12}.
$$

\section{Nonzero bracket relations on $(\mathfrak{g}_2)_\mathbb{C}$}
\label{AppendixG2Complex}
Let $E_{\alpha_j}$ for $j=1,\dots, 6$ be defined by (\ref{eqEalpha1})-(\ref{eqEalpha6}) and let   
$$
H_1:=b_1,~\hspace*{0.1in}H_2:=-\frac{1}{\sqrt{3}}b_1-\frac{2}{\sqrt{3}}b_{12}.
$$
The nonzero bracket relations on $(\mathfrak{g}_2)_{\mathbb{C}}$ with respect to the basis
$$
H_1,~H_2, E_{\alpha_1},\dots, E_{\alpha_6},~\overline{E}_{\alpha_1},\dots, \overline{E}_{\alpha_6}
$$
are given as follows:  
$$
[H_1,E_{\alpha_1}]=iE_{\alpha_1},~[H_1,E_{\alpha_2}]=iE_{\alpha_2},~[H_1,E_{\alpha_3}]=iE_{\alpha_3},~[H_1,E_{\alpha_4}]=iE_{\alpha_4},
$$
$$
[H_1,E_{\alpha_5}]=2iE_{\alpha_5},~[H_1,\overline{E}_{\alpha_1}]=-i\overline{E}_{\alpha_1},~[H_1,\overline{E}_{\alpha_2}]=-i\overline{E}_{\alpha_2}
$$
$$
[H_1,\overline{E}_{\alpha_3}]=-i\overline{E}_{\alpha_3},~[H_1,\overline{E}_{\alpha_4}]=-i\overline{E}_{\alpha_4},~[H_1,\overline{E}_{\alpha_5}]=-2i\overline{E}_{\alpha_5},
$$
$$
[H_{2},E_{\alpha_1}]=-i\sqrt{3}E_{\alpha_1},~[H_{2},E_{\alpha_2}]=i\sqrt{3}E_{\alpha_2},~[H_2,E_{\alpha_3}]=-\frac{i}{\sqrt{3}}E_{\alpha_3},
$$
$$
[H_{2},E_{\alpha_4}]=\frac{i}{\sqrt{3}}E_{\alpha_4},~[H_{2},E_{\alpha_6}]=\frac{2i}{\sqrt{3}}E_{\alpha_6},~[H_2,\overline{E}_{\alpha_1}]=i\sqrt{3}\overline{E}_{\alpha_1},
$$
$$
[H_{2},\overline{E}_{\alpha_2}]=-i\sqrt{3}\overline{E}_{\alpha_2},~[H_2,\overline{E}_{\alpha_3}]=\frac{i}{\sqrt{3}}\overline{E}_{\alpha_3},~[H_{2},\overline{E}_{\alpha_4}]=-\frac{i}{\sqrt{3}}\overline{E}_{\alpha_4},
$$
$$
[H_{2},\overline{E}_{\alpha_6}]=-\frac{2i}{\sqrt{3}}\overline{E}_{\alpha_6},~[E_{\alpha_1},E_{\alpha_2}]=2E_{\alpha_5},~[E_{\alpha_1},E_{\alpha_6}]=-2iE_{\alpha_3},
$$
$$
[E_{\alpha_1},\overline{E}_{\alpha_1}]=2iH_1-2i\sqrt{3}H_2,~[E_{\alpha_1},\overline{E}_{\alpha_3}]=2i\overline{E}_{\alpha_6},
$$
$$
[E_{\alpha_1},\overline{E}_{\alpha_5}]=-2\overline{E}_{\alpha_2},~[E_{\alpha_2},\overline{E}_{\alpha_2}]=2iH_1+2i\sqrt{3}H_2,~[E_{\alpha_2},\overline{E}_{\alpha_4}]=2iE_{\alpha_6},
$$
$$
[E_{\alpha_2},\overline{E}_{\alpha_5}]=2\overline{E}_{\alpha_1},~[E_{\alpha_2},\overline{E}_{\alpha_6}]=-2iE_{\alpha_4},~[E_{\alpha_3},E_{\alpha_4}]=-6E_{\alpha_5},
$$
$$
[E_{\alpha_3},E_{\alpha_6}]=4iE_{\alpha_4},~[E_{\alpha_3},\overline{E}_{\alpha_1}]=2iE_{\alpha_6},
$$
$$
[E_{\alpha_3},\overline{E}_{\alpha_3}]=6iH_1-2i\sqrt{3}H_2,~[E_{\alpha_3},\overline{E}_{\alpha_4}]=-4i\overline{E}_{\alpha_6},
$$
$$
[E_{\alpha_3},\overline{E}_{\alpha_5}]=2\overline{E}_{\alpha_4},~[E_{\alpha_3},\overline{E}_{\alpha_6}]=-6iE_{\alpha_1},~[E_{\alpha_4},E_{\alpha_6}]=-6iE_{\alpha_2},
$$
$$
[E_{\alpha_4},\overline{E}_{\alpha_2}]=2i\overline{E}_{\alpha_6},~[E_{\alpha_4},~\overline{E}_{\alpha_3}]=-4iE_{\alpha_6}
$$
$$
[E_{\alpha_4},\overline{E}_{\alpha_4}]=6iH_1+2i\sqrt{3}H_2,~[E_{\alpha_4},\overline{E}_{\alpha_5}]=-2\overline{E}_{\alpha_3},
$$
$$
[E_{\alpha_4},\overline{E}_{\alpha_6}]=4iE_{\alpha_3},~[E_{\alpha_5},\overline{E}_{\alpha_1}]=2E_{\alpha_2},~[E_{\alpha_5},\overline{E}_{\alpha_2}]=-2E_{\alpha_1},
$$
$$
[E_{\alpha_5},\overline{E}_{\alpha_3}]=-2E_{\alpha_4},~[E_{\alpha_5},\overline{E}_{\alpha_4}]=2E_{\alpha_3},~[E_{\alpha_5},\overline{E}_{\alpha_5}]=4iH_1,
$$
$$
[E_{\alpha_6},\overline{E}_{\alpha_2}]=-2i\overline{E}_{\alpha_4},~[E_{\alpha_6},\overline{E}_{\alpha_3}]=-6i\overline{E}_{\alpha_1},~[E_{\alpha_6},\overline{E}_{\alpha_4}]=4i\overline{E}_{\alpha_3},~
$$
$$
[E_{\alpha_6},\overline{E}_{\alpha_6}]=4i\sqrt{3}H_2,~[\overline{E}_{\alpha_1},\overline{E}_{\alpha_2}]=2\overline{E}_{\alpha_5},~[\overline{E}_{\alpha_1},\overline{E}_{\alpha_6}]=2i\overline{E}_{\alpha_3},
$$
$$
[\overline{E}_{\alpha_3},\overline{E}_{\alpha_4}]=-6\overline{E}_{\alpha_5},~[\overline{E}_{\alpha_3},\overline{E}_{\alpha_6}]=-4i\overline{E}_{\alpha_4},~[\overline{E}_{\alpha_4},\overline{E}_{\alpha_6}]=6i\overline{E}_{\alpha_2}.
$$

\end{document}